\documentclass[11pt,a4paper,oneside]{amsart}
\pdfoutput=1
\usepackage{amssymb}
\usepackage{esint}
\usepackage{mathtools}
\usepackage{fullpage}
\usepackage{hyperref}
\usepackage{enumitem}

\usepackage[english]{babel}
\usepackage[utf8]{inputenc} 
\usepackage[pdftex]{graphicx}
\usepackage{xcolor}

\usepackage{graphicx}
\usepackage{subcaption}
\usepackage{natbib}
\usepackage{multirow}

\theoremstyle{plain}
\newtheorem{theorem}{Theorem}[section]

\newtheorem{corollary}[theorem]{Corollary}
\newtheorem{proposition}[theorem]{Proposition}

\theoremstyle{definition}

\theoremstyle{remark}

\renewcommand{\phi}{\varphi}

\title[Torus computed tomography for experimental data]{Torus computed tomography for experimental data}

\author{Ella Salo}
\address{Computational Engineering, School of Engineering Sciences,
Lappeenranta--Lahti University of Technology LUT, Lappeenranta, Finland}
\email{ella.salo@student.lut.fi}

\author{Alexander Meaney}
\address{Department of Mathematics and Statistics, University of Helsinki, Finland}
\email{alexander.meaney@helsinki.fi}

\author{Olli Koskela}
\address{HAMK Tech research unit, Häme University of Applied Sciences, Hämeenlinna, Finland}
\email{olli.koskela@hamk.fi}

\author{Jesse Railo}
\address{Computational Engineering, School of Engineering Sciences,
Lappeenranta--Lahti University of Technology LUT, Lappeenranta, Finland}
\email{jesse.railo@lut.fi}

\date{\today}

\begin{document}
\begin{abstract}
Torus computed tomography (Torus CT) reconstructs the Fourier coefficients of an image from X-ray data along closed geodesics on the flat torus. We extend Torus CT to experimental X-ray tomographic data by converting the data to a parallel-beam geometry and mapping the measured projection angles to the geodesic directions on the torus. In addition, we introduce two extensions of the Torus CT framework: the Star TCT method, which recovers additional Fourier coefficients beyond the standard frequency box using the same data, and torus backprojection (TBP), which sums the data over the measured directions. We prove that filtered TBP converges to the Tikhonov-regularized solution for general data up to an explicit error term that vanishes for data which is invariant along fixed geodesics. We also show that the direction-wise averaging of the zero mode is an unbiased estimator and its mean squared error decays like $N^{-2}$. We evaluate the methods on simulated phantoms and experimental X-ray data of a walnut using an updated implementation. The results show that torus-based reconstruction methods are applicable to experimental data, avoid the streak artifacts of filtered backprojection (FBP), and yield competitive reconstruction errors on both simulated and experimental data.
\end{abstract}


\maketitle

\section{Introduction}
\label{sec:intro}

Computed tomography (CT) allows recovering cross-sectional images of objects from their projections, usually acquired using X-ray imaging. The mathematical foundations date back to the work of Radon in 1917 \cite{Radon_1917__Berichte, Radon_1986__IEEETransMedImaging}. Notably, it took over half a century to realize the significance of the Radon transform and its inversion to tomographic reconstruction through the work of Cormack \cite{C63,C64,Cormack_1973__PhysMedBiol} and Hounsfield \cite{H73}. The standard method for tomographic reconstruction is filtered backprojection (FBP) \cite{Kak_Slaney_1988, N01}. Subsequently, numerous different approaches have been developed for CT reconstruction, broadly categorizable into analytical (\emph{i.e.} direct) methods, iterative methods which tackle the reconstruction problem as a large linear system, and methods based on machine learning \cite{Clackdoyle_Defrise_2010__IEEESignalProcessMag,Kak_Slaney_1988, Ravishankar_et_al_2020__ProcIEEE}. In recent years, algorithm development has focused substantially on obtaining reconstructions from noisy and/or incomplete data, where FBP produces inadequate results \cite{Pan_et_al_2009__InverseProbl, Ravishankar_et_al_2020__ProcIEEE}.

Many reconstruction methods are based on the Fourier slice theorem, which identifies the Fourier transform of a projection with a radial slice of the Fourier transform of the object. In the Euclidean setting, these slices do not match a Cartesian frequency grid, so direct Fourier methods require interpolation \cite{N01}, which has motivated linogram and pseudo-polar constructions with non-uniform slope-based sets of directions \cite{Averbuch08, EPH87}.

On the flat torus $\mathbb{T}^2:=\mathbb{R}^2/\mathbb{Z}^2$, the Fourier slice theorem takes an exact discrete form for closed geodesics. The X-ray data in a rational direction $v$ determines exactly the Fourier coefficient at the frequencies orthogonal to $v$. Thus, the mismatch does not occur. The geodesic X-ray transform on the torus has been extensively studied in theoretical contexts. Ilmavirta established injectivity and reconstruction results for the $d$-plane Radon transform on $\mathbb{T}^n$ in \cite{I15}, extending earlier work by Strichartz \cite{S82}, and Abouelaz and Rouvi\`ere \cite{AR11}. Fourier-based reconstruction formulas, stability estimates, and Tikhonov regularization for periodic Radon transforms were further developed by Railo in \cite{R20}.

The Torus CT method, introduced by Ilmavirta et al. in \cite{IKR19}, establishes a connection between the geodesic X-ray transform on a flat torus and practical CT imaging. When the imaged object is supported in the unit square and extended periodically, a closed geodesic in an integer direction $v$ corresponds to a finite union of parallel line segments in the unit square (see Figure \ref{fig:segments}). Consequently, a single parallel-beam projection at the angle of $v$ contains all line integrals required to evaluate the geodesic X-ray transform in that direction. The original work uses this information to recover Fourier coefficients of the object in a finite frequency box from a finite set of measurement directions. The resulting spectral representation of the target allows explicit control over recoverable frequency content and enables regularization directly in the Fourier domain. The reconstruction is also meshless and once the Fourier coefficients are computed, it can be evaluated at any point. While \cite{IKR19} established the theoretical framework and provided numerical demonstrations, it remained unclear how effectively the method could be applied to practical X-ray data.

This study applies Torus CT to experimental fan-beam X-ray measurements of a walnut \cite{walnut, walnut_arxiv}. Because the data are acquired at equally spaced projection angles rather than the closed geodesic directions on the torus required by the Torus CT,  the fan-beam sinogram is converted to parallel-beam Radon data and the directions are approximated by nearest-angle selection. A few shortcomings of the numerical implementation of \cite{IKR19} were corrected and the implementation was extended to accommodate experimental data. The resulting reconstructions are compared with filtered backprojection using the same set of torus-optimal measurement directions.
With the corrected implementation, we also revisit the simulated-data experiments in \cite{IKR19} and obtain substantially improved results. 

A key advantage of the torus based inversion is that the same measurement data can be utilized in multiple ways. Building on this observation, we introduce two extensions of the original Torus CT approach. The Star TCT method extends the recoverable frequency coverage by recovering additional Fourier coefficients beyond the standard frequency box without requiring any additional data. Torus backprojection (TBP) is based on the inversion formula of \cite[Theorem 1.3]{R20}, which reconstructs the object as a sum of its X-ray data over all directions. We prove that convolving the backprojection of general data with an explicit periodic kernel converges to the Tikhonov regularized solution up to an error term that measures how far the data are from being constant along geodesics. This error term vanishes for any data defined geodesic by geodesic, including noisy and experimental data. We further show that the direction-wise averaging used to estimate the zero Fourier mode in the implementation is unbiased, with a mean squared error that decays as $N^{-2}$.

Compared with FBP using the same projection directions, the torus-based methods avoid directional streak artifacts \cite{Q93,Q17} and yield competitive reconstruction error on both simulated and experimental data.

The main contributions of this work are as follows:
\begin{itemize}
    \item A filtered torus backprojection that converges to the Tikhonov regularized solution for general data, with an explicit error term that vanishes for data defined along geodesics, including noisy and experimental data (Theorem \ref{thm:filtered-TBP}, Corollaries \ref{cor:error_estimate}--\ref{cor:noisydata}).
    \item An analysis of the direction-wise averaging of the zero mode, which is unbiased with mean squared error decaying like $N^{-2}$ (Proposition \ref{prop:zeromode}).
    \item Introduction of the Star TCT method for extended frequency recovery.
    \item Extension of Torus CT to experimentally acquired fan-beam X-ray data, and a revised and extended implementation.
    \item A quantitative comparison with FBP on simulated and experimental data, including the effect of a positivity constraint.
\end{itemize}

The structure of the paper is as follows. Section \ref{sec:methods} recalls the theoretical framework of Torus CT following \cite{IKR19} and presents the extensions developed in this work. Section \ref{sec:implementation} describes the numerical implementation of the methods, including Star TCT and torus backprojection and its regularization. Sections \ref{sec:simulated} and \ref{sec:realdata} present numerical results using simulated and real data, respectively, and a comparison with filtered backprojection. Conclusions are given in Section \ref{sec:conclusions}.

\section{Torus CT method}
\label{sec:methods}

In this section, we recall the theoretical framework of the Torus CT method as presented in \cite{IKR19}, which forms the basis for our numerical implementation with both simulated and real data. In addition, we introduce the torus backprojection (TBP) method following \cite{R20}, and prove the convergence of filtered TBP for general data.

The reconstruction method is based on the Fourier series representation of functions and the geodesic X-ray transform on $\mathbb{T}^2$. The presentation of the methodology is limited to definitions and formulas used directly in the implementation. The core theoretical results related to Torus CT are recalled from \cite{IKR19}, while the TBP formulation follows \cite{R20}.

\subsection{X-ray transform on $\mathbb{T}^2$}
Let the flat two-dimensional torus be defined as the quotient $\mathbb{T}^2:=\mathbb{R}^2/\mathbb{Z}^2$ and the quotient mapping denoted by $\left[\cdot \right]:\mathbb{R}^2\rightarrow \mathbb{T}^2$. Functions on $\mathbb{T}^2$ can be identified with $\mathbb{Z}^2$-periodic functions on $\mathbb{R}^2$ via the quotient mapping.

The geodesic X-ray transform on closed Riemannian manifolds is defined as a collection of line integrals over periodic geodesics. On $\mathbb{T}^2$, all geodesics are parametrized by $\gamma_{x,v}(t):=[x+tv], x\in[0,1]^2, v\in\mathbb{R}^2\setminus\{0\}$. Such a geodesic is periodic if and only if $v$ is a multiple of a rational vector. In particular, if $v\in\mathbb{Z}^2\setminus\{0\}$, then $\gamma_{x,v}$ is periodic with a period one.

Let $\mathcal{T}:=C^\infty(\mathbb{T}^2)$ denote the space of test functions and $\mathcal{T}'$ its dual. The geodesic X-ray transform on $\mathbb{T}^2$ is defined for $f\in\mathcal{T}$ by
\begin{equation}\label{eq:If}
    \mathcal{I}f(x,v):=\int_0^1 f(\gamma_{x,v}(t))~dt, \quad x\in [0,1]^2, \quad v\in\mathbb{Z}^2\setminus\{0\}.
\end{equation}
For each fixed direction $v$, the operator $f\mapsto \mathcal{I}f(\cdot,v)$ is formally self-adjoint on $\mathcal{T}$. By duality, the definition extends to distributions $f\in\mathcal{T}'$ as described in \cite{IKR19}.

For every $v\in\mathbb{Z}^2\setminus \{0\}$ and $t\in\mathbb{R}$, the X-ray data are constant along the geodesics \begin{equation}
\label{eq:periodicity_IF}
\mathcal{I}f(x+tv,v)=\mathcal{I}f(x,v), \quad x\in\mathbb{T}^2,\  t\in \mathbb{R},
\end{equation}
which follows directly from \eqref{eq:If}, since $\gamma_{x,v}$ is periodic with period one. See also \cite[Lemma 7]{IKR19}.

\subsection{Fourier representation and inversion on $\mathbb{T}^2$}
For $f\in\mathcal{T}'$, the Fourier coefficients are defined by $\hat{f}(k):=f(e^{-2\pi i k\cdot x})$, $k\in\mathbb{Z}^2$, and the Fourier series
\begin{equation}
    f(x) = \sum_{k\in\mathbb{Z}^2}\hat{f}(k)e^{2\pi i k\cdot x}
\end{equation}
converges in the sense of distributions.

The inversion of the geodesic X-ray transform on $\mathbb{T}^2$ is based on the Fourier slice theorem.
\begin{theorem}[Fourier slice theorem {\cite[Theorem 5]{IKR19}}]\label{thm:FourierSlice} Let $f\in\mathcal{T}'$ and $v\in\mathbb{Z}^2\setminus\{0\}$. The Fourier transform of the X-ray data satisfies
\begin{equation}
    \widehat{\mathcal{I}f}(k,v) = \begin{cases}
\hat{f}(k),&k\cdot v = 0,\\
0,&k\cdot v  \neq 0.
\end{cases}
\end{equation}
\end{theorem} 
As a consequence, each Fourier coefficient of $f$ can be recovered from X-ray data taken in a direction orthogonal to the corresponding frequency. The following theorem summarizes the explicit reconstruction formula and is the main theoretical result underlying the inversion.
\begin{theorem}[{ \cite[Theorem 1]{IKR19}}]
\label{thm:inversion}
    Suppose that $f\in L^1(\mathbb{T}^2)$. Let $k\in\mathbb{Z}^2$. If $k\neq 0$ and $v\in\mathbb{Z}^2\setminus\{0\}$ satisfies $v\perp k$, then
    \begin{equation}
        \hat{f}(k)=\begin{cases}
            \int_0^1\mathcal{I}_v f(0,y)\exp{(-2\pi i k_2 y)}~dy, & k_2\neq 0,\\
            \int_0^1\mathcal{I}_v f(x,0)\exp{(-2\pi i k_1 x)}~dx, & k_1\neq 0.
        \end{cases}
    \end{equation}
    If $k=0$, then
    \begin{equation}
        \label{eq:mean}\hat{f}(0)=\int_0^1\mathcal{I}_{(1,0)}f(0,y)~dy=\int_0^1\mathcal{I}_{(0,1)}f(x,0)~dx.
    \end{equation}
\end{theorem}
The function $f$ is reconstructed by its Fourier series using the coefficients obtained from the above formulas.

\subsection{Adjoint and regularization}
\label{subsec:adjoint_and_reg}
Let $Q\subset \mathbb{Z}^2$ be the set of primitive integer directions, chosen such that every non-zero $v\in \mathbb{Z}^2$ is an integer multiple of a unique element of $Q$
\begin{equation}
    Q := \{ v\in \mathbb{Z}^2 : \gcd(v_1,v_2)=1,\ \text{and } v_1>0 \text{ or } v=(0,1)\}. 
\end{equation}
The geodesic X-ray transform maps functions on $\mathbb{T}^2$ to functions on $\mathbb{T}^2\times Q$.

For $s\in\mathbb{R}$, we study functions in Sobolev spaces $H^s(\mathbb{T}^2)$ with norm 
\begin{equation}
    \| f \|^2_{H^s(\mathbb{T}^2)}=\sum_{k\in\mathbb{Z}^2}\langle k \rangle^{2s}| \hat{f}(k) |^2,
\end{equation}
where $\langle k \rangle = (1+|k|^2)^{\frac{1}{2}}$. The data space $H^s(\mathbb{T}^2\times Q)$ consists of functions $g$ such that $g(\cdot,v)\in H^s(\mathbb{T}^2)$ for every $v\in Q$ and the average of $g(\cdot,v)$ over $\mathbb{T}^2$ is the same for all $v\in Q$, with the norm
\begin{equation}
    \| g \|^2_{H^s(\mathbb{T}^2\times Q)} = |\hat{g}(0,0)|^2 + \sum_{k\in\mathbb{Z}^2\setminus 0} \sum_{v\in Q} \langle k\rangle ^{2s}| \hat{g}(k,v) |^2.
\end{equation}
For any $s\in\mathbb{R}$, the adjoint of $\mathcal{I}:H^s(\mathbb{T}^2)\rightarrow H^s(\mathbb{T}^2\times Q)$ is given by 
\begin{equation}
\label{eq:adjoint}
    \widehat{\mathcal{I}^*g}(k)=\hat{g}(k,\hat{k}^\perp), \qquad k\in \mathbb{Z}^2,
\end{equation}
where $\hat{k}^{\perp}$ denotes the unique element of $Q$ orthogonal to $k$. The normal operator satisfies $\mathcal{I}^*\mathcal{I}=\mathrm{Id}$ \citep[Proposition 11]{IKR19}.

To stabilize the inversion in the presence of noise, we use Tikhonov regularization. Let $r\in \mathbb{R}$ and $s\geq r$. Given data $g\in H^r(\mathbb{T}^2\times Q)$ and $\alpha>0$, we define 
\begin{equation}
\label{eq:tikhonov}
    f^{\alpha,s} = \arg \min_{f\in H^r(\mathbb{T}^2)}(\| \mathcal{I}f - g \|^2_{H^r(\mathbb{T}^2\times Q)} + \alpha \| f \|^2_{H^s(\mathbb{T}^2)}).
\end{equation}
The unique minimizer is $f^{\alpha,s} = P_\alpha^{s-r}\mathcal{I}^*g$, with
\begin{equation}
\widehat{P_\alpha^{s-r}h}(k)=\frac{1}{1+\alpha \langle k \rangle^{2(s-r)}}\hat{h}(k),
\end{equation}
satisfying $f^{\alpha,s}\in H^{2s-r}(\mathbb{T}^2)\subset H^r(\mathbb{T}^2)$ \citep[Theorem 2]{IKR19}.

\subsection{Discrete directions}

The reconstruction is done in the Fourier domain by recovering the Fourier coefficients $\hat{f}(k)$ for frequencies $k\in\mathbb{Z}^2$ contained in a finite box $\mathbb{Z}^2_N:=[-N,N]^2\cap \mathbb{Z}^2$. By the Fourier slice theorem, each coefficient $\hat{f}(k)$ can be recovered from X-ray data taken in direction $v\in \mathbb{Z}^2\setminus\{ 0\}$ satisfying $v\perp k$. Therefore, the set of discrete measurement directions must be chosen so that for every $k\in\mathbb{Z}_N^2$ there exists at least one integer direction orthogonal to $k$. We define the set of admissible directions by
\begin{equation}
    A_N:=\{ v\in\mathbb{Z}^2\setminus\{0\}:v\in k^\perp ~\mathrm{for~ some} ~k\in \mathbb{Z}_N^2\setminus\{0\}\}.
\end{equation}
Since $\mathcal{I}_{mv}f=\mathcal{I}_v f$ for all $m\in\mathbb{Z}\setminus\{0\}$, we only need one primitive representative per direction. Therefore, we define the set $Q_N$ as
\begin{equation}
    Q_N := Q\cap A_N.
\end{equation} 
This guarantees that all Fourier coefficients in $\mathbb{Z}_N^2$ are determined while keeping the number of projection directions minimal. The directions are illustrated in Figure \ref{fig:angles}.

\begin{figure}[h]
    \centering

    \begin{subfigure}{0.2\textwidth}
        \centering
        \includegraphics[width=\linewidth]{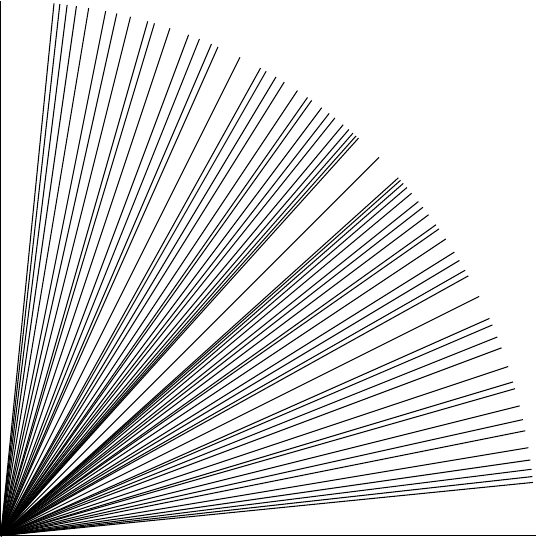}
        \caption{}
    \end{subfigure}
    \hspace{0.04\textwidth}
    \begin{subfigure}{0.2\textwidth}
        \centering
        \includegraphics[width=\linewidth]{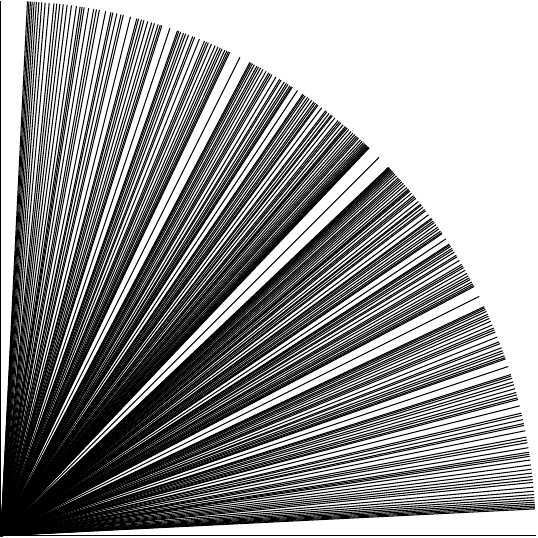}
        \caption{}
    \end{subfigure}
    \hspace{0.04\textwidth}
    \begin{subfigure}{0.2\textwidth}
        \centering
        \includegraphics[width=\linewidth]{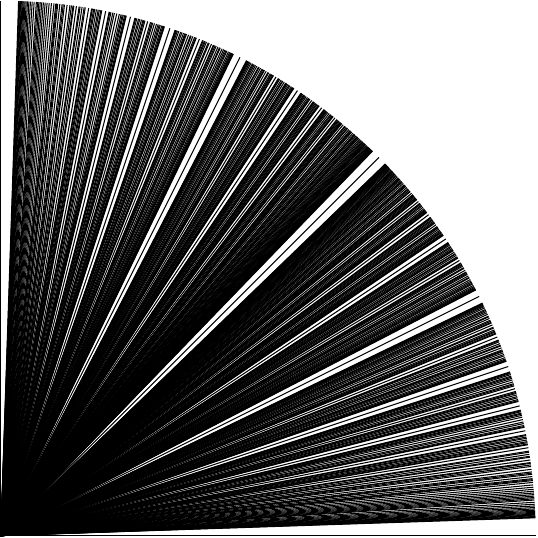}
        \caption{}
    \end{subfigure}

    \caption{Discrete measurement directions $Q_N$ for (A) $N=10$, (B) $N=20$ and (C) $N=30$ between $0^\circ$ and $90^\circ$. The directions are sparse near the axial and diagonal directions.}
    \label{fig:angles}
\end{figure}

\subsection{Torus backprojection}

The following inversion formula by \cite{R20} provides an alternative reconstruction method which is based on summing the X-ray data directly over all measurement directions.

\begin{theorem}[{\cite[Theorem 1.3]{R20}}]
\label{thm:TBP}Let $f\in \mathcal{T}'$ with $\hat{f}(0)=0$. Then \begin{equation}
    f = \sum_{v\in Q} \mathcal{I}_v f,
\end{equation}
where sum is taken over all primitive integer directions $v\in Q$ and $\mathcal{I}_vf:=\mathcal{I}f(\cdot,v)$ denotes the X-ray data in direction $v$.
\end{theorem}

The formula states that $f$ is recovered exactly as a sum of its X-ray projections over all primitive directions, without any Fourier analysis or filtering. We refer to this method as TBP (torus backprojection). The zero-average condition $\hat{f}(0)=0$ eliminates the zero Fourier mode from the reconstruction formula. This is necessary in the infinite sum, since each term $\mathcal{I}_v f$ contributes the zero mode, and summing over all directions would otherwise result in an infinite contribution.

In the discrete setting, the infinite sum is approximated by a finite set of directions $Q_N\subset Q$, and the mean value can be subtracted $(|Q_N|-1)$ times, leaving the correct single contribution. The mean-value correction is incorporated into the sum as
\begin{equation}
\label{eq:f_N_TBP}
    f_{N}(x)=\sum_{v\in Q_N}(\mathcal{I}_vf(x) - \frac{|Q_N|-1}{|Q_N|}\hat{f}(0)),
\end{equation}
where $\hat{f}(0)$ is computed from the data as in \eqref{eq:mean}.

Next, we state and prove a new convergence and reconstruction result related to this computationally relevant solution to the inverse problem at hand.
\begin{proposition}
\label{prop:f_N_convergence}
    Let $f\in\mathcal{T}'$. Then $f_N$ converges to $f$ in the sense of distributions, i.e.,
    \begin{equation}
        \lim_{N\rightarrow \infty} f_N = f \quad \text{in } \mathcal{T'}.
    \end{equation}
    Moreover, if $f\in H^t(\mathbb{T}^2)$ for some $t\in \mathbb{R}$, then $f_N\rightarrow f$ in $H^t(\mathbb{T}^2)$.
\end{proposition}
\begin{proof} By the Fourier slice theorem~\ref{thm:FourierSlice}, it holds that 
\begin{equation}
    \mathcal{I}_vf(x) = \sum_{v \cdot k=0, \,k \in \mathbb{Z} ^2} \hat{f}(k)e^{2\pi i k \cdot x}
\end{equation}
for any $v \in \mathbb{Z}^2 \setminus \{0\}$, in the sense of distributions, with a slight abuse of notation. Furthermore, from the fact $f \in \mathcal{T}'$, one notices that $\mathcal{I}_vf \in \mathcal{T}'$ for any $v \in \mathbb{Z}^2 \setminus \{0\}$. Since $Q_N$ is a finite set, it follows that $f_N \in \mathcal{T}'$ as a finite sum of distributions for any $N \in \mathbb{N}$.

Let now $k \in \mathbb{Z}^2 \setminus \{0\}$. We show that there exists $N_k \geq 0$ such that for any $N \geq N_k$ it holds that $\hat{f}_N(k)=\hat{f}(k)$. It follows from the definition of $Q_N$ that for any $N \geq |k|_\infty$ there exists $v \in Q_N$ such that $v \cdot k = 0$. One may now compute, using the Fourier slice theorem~\ref{thm:FourierSlice} and the definition of $f_N$, that $\hat{f}_N(k)=\hat{f}(k)$ for any $N \geq N_k := |k|_{\infty}$. On the other hand, one can compute that 
\begin{equation}
    \hat{f}_N(0) = \sum_{v \in Q_N} (\widehat{\mathcal{I}_vf}(0) - \frac{|Q_N|-1}{|Q_N|}\hat{f}(0)) = \sum_{v \in Q_N} (\hat{f}(0) - \frac{|Q_N|-1}{|Q_N|}\hat{f}(0)) = \hat{f}(0)
\end{equation}
for any $N \in \mathbb{N}$. In particular, $\hat{f}_N(k) = \hat{f}(k)$ for all $k\in\mathbb{Z}^2_N$.

Since $f\in\mathcal{T}'$, there exists $t\in\mathbb{R}$ such that $f\in H^t(\mathbb{T}^2)$. For $\xi\neq 0$, at most one $v\in Q_N$ is orthogonal to $\xi$, so $\hat{f}_N(\xi)$ attains either value $\hat{f}(\xi)$ or zero. We can now estimate
\begin{align}
    \| f_{N} - f \|^2_{H^t(\mathbb{T}^2)} &= \sum _{\xi \in \mathbb{Z}^2} \langle  \xi\rangle^{2t} | (\hat{f}_{N} - \hat{f})(\xi)|^2 = \sum_{\xi\in\mathbb{Z}^2\setminus \mathbb{Z}^2_{N}} \langle  \xi\rangle^{2t} | (\hat{f}_{N} - \hat{f})(\xi)|^2\\ &\leq \sum_{\xi\in\mathbb{Z}^2\setminus \mathbb{Z}^2_{N}} \langle  \xi\rangle^{2t} | \hat{f}(\xi)|^2 \rightarrow 0, \quad \text{as } N \rightarrow \infty.
\end{align}

Hence, the limit $[\lim_{N\to\infty}f_N](\phi) := \lim_{N \to \infty} \langle f_N, \phi \rangle$, where $\phi \in \mathcal{T}$ is a test function, defines a distribution in $\mathcal{T}'$, which agrees with $f$ as they have the same Fourier series coefficients.
\end{proof}

We next state and prove the corresponding approximation result for the Tikhonov regularized reconstruction method associated with $f_N$ for given data $g\in H^r(\mathbb{T}^2\times Q)$. By replacing $\mathcal{I}f$ with the data, we acquire the formula 
\begin{equation}
\label{eq:g_N_TBP}
    g_{N}(x)=\sum_{v\in Q_N}(g(x,v) - \frac{|Q_N|-1}{|Q_N|}\hat{g}(0,0)).
\end{equation}

The data space $H^r(\mathbb{T}^2\times Q)$ does not require $g(\cdot,v)$ to be constant along the geodesic in direction $v$. We measure the deviation from this property by
\begin{equation}
\label{eq:error_e}
    e(t;x,v):=g(x+tv,v)-g(x,v), \qquad t\in\mathbb{R},x\in\mathbb{T}^2, v\in Q.
\end{equation}
For an exact data $g=\mathcal{I}f$ and any X-ray data constructed geodesic by geodesic, including the noisy and experimental torus data of Sections \ref{sec:simulated} and \ref{sec:realdata}, we have $e\equiv0$ by \eqref{eq:periodicity_IF}. On the Fourier side, we have 
\begin{equation}
    \hat{e}(t;k,v)=(e^{2\pi i tk\cdot v }-1)\hat{g}(k,v).
\end{equation}
If $k\cdot v\neq 0$, choosing $t=\frac{1}{2k\cdot v}$ yields $\hat{g}(k,v)=-\frac{1}{2}\hat{e}(\frac{1}{2k\cdot v};k,v)$. This means coefficients $\hat{g}(k,v)$ with $k\cdot v\neq 0$ vanish when $e\equiv 0$. We collect them into a Fourier side error term $E_N$ defined by
\begin{equation}
\label{eq:errorE}
    \hat{E}_N(k):=\sum_{v\in Q_N, k\cdot v\neq 0} \hat{g}(k,v), \quad k\in\mathbb{Z}^2,
\end{equation}
so that $E_N=0$ whenever $e\equiv 0$. Since the sum has at most $|Q_N|$ terms, the Cauchy--Schwarz inequality gives $\|E_N\|^2_{H^r(\mathbb{T}^2)}\leq |Q_N|\sum_{v\in Q_N}\| g(\cdot,v) \|_{H^r(\mathbb{T}^2)}^2 < \infty$, so $E_N\in H^r(\mathbb{T}^2)$.

\begin{theorem}[Filtered torus backprojection]\label{thm:filtered-TBP}
    Let $r\in\mathbb{R}, s\geq r$, $\alpha>0$ and let $f^{\alpha,s}$ denote the Tikhonov regularized solution corresponding to \eqref{eq:tikhonov} with data $g\in H^r(\mathbb{T}^2\times Q)$. Then
    \begin{equation}
        f^{\alpha,s}=\lim_{N\to\infty} ((g_N -E_N) * p_{\alpha}^{\,s-r})
    \quad \text{in } H^{2s-r}(\mathbb{T}^2),
    \end{equation}
    where $g_N$ and $E_N$ are defined as in \eqref{eq:g_N_TBP} and \eqref{eq:errorE}, and
    \begin{equation}
        p_{\alpha}^{s-r}(x)=\sum_{k\in\mathbb{Z}^2}\frac{1}{1+\alpha\langle k\rangle^{2(s-r)}}e^{2\pi i k\cdot x}.
    \end{equation}
    In particular, if $e\equiv 0$, then $f^{\alpha,s}=\lim_{N\rightarrow\infty}(g_N*p_\alpha^{s-r})$ in $H^{2s-r}(\mathbb{T}^2)$.
\end{theorem}
\begin{proof}
    By \cite[Theorem 2]{IKR19}, the minimizer of \eqref{eq:tikhonov} is $f^{\alpha,s}=P_{\alpha}^{s-r}\mathcal{I}^*g$. 
    
    We next identify $P_{\alpha}^{s-r}$ as a convolution operator. By definition
    \begin{equation}
        \widehat{f^{\alpha,s}}(k) = \frac{1}{1+\alpha\langle k\rangle^{2(s-r)}}\widehat{\mathcal{I}^*g}(k) = \widehat{p_{\alpha}^{s-r}}(k)\,\widehat{\mathcal{I}^*g}(k).
    \end{equation}

    Now, the convolution theorem implies
    \begin{equation}
        f^{\alpha,s} = \mathcal{I}^*g * p_{\alpha}^{s-r}.
    \end{equation}

    Next, we show that $g_N-E_N\rightarrow\mathcal{I}^*g$ in $H^r(\mathbb{T}^2)$. On the Fourier side, when $k\neq0$, we can split directions according to whether they are orthogonal, which gives
    \begin{equation}
\hat{g}_N(k)=\sum_{v\in Q_N, k\cdot v= 0} \hat{g}(k,v) + \hat{E}_N(k).
    \end{equation}
    
    The only direction orthogonal to $k$ in $Q$ is $\hat{k}^\perp$ which is in $Q_N$ whenever $N\geq |k|_{\infty}$. Now by \eqref{eq:adjoint}, we have, $(g_N-E_N)^{\wedge}(k)=\hat{g}(k,\hat{k}^\perp)=\widehat{\mathcal{I}^*g}(k)$ for $N\geq|k|_{\infty}$, and $(g_N-E_N)^{\wedge}(k)\in\{0,\widehat{\mathcal{I}^*g}(k)\}$ for all $N$. For $k=0$, we have $\hat{E}_N(0)=0$, and all averages of $g$ equal $\hat{g}(0,0),$ so $(g_N-E_N)^{\wedge}(0)=\hat{g}(0,0)=\widehat{\mathcal{I}^*g}(0)$. Therefore 
    \begin{equation}
        \| g_N - E_N- \mathcal{I}^*g \|^2_{H^r(\mathbb{T}^2)} \leq \sum_{k\in \mathbb{Z}^2\setminus \mathbb{Z}_N^2} \langle k \rangle^{2r} |\widehat{\mathcal{I}^*g}(k)|^2 \rightarrow 0 \quad \text{as }\,N \rightarrow\infty,
    \end{equation}
    since $\mathcal{I}^*g\in H^r(\mathbb{T}^2)$. In particular $g_N-E_N\rightarrow \mathcal{I}^*g$ in $H^r(\mathbb{T}^2)$.

    The filter multiplies the $k$'th Fourier coefficient by $(1+\alpha\langle k \rangle^{2(s-r)})^{-1}$, and
    \begin{equation}
    \label{eq:filterupperbound}
        \frac{\langle k \rangle^{2(s-r)}}{1+\alpha \langle k \rangle^{2(s-r)}}\leq \frac{1}{\alpha}, \qquad k\in\mathbb{Z}^2.
    \end{equation}
    Therefore, for every $h\in H^{r}(\mathbb{T}^2)$, one has $\| h*p_{\alpha}^{s-r} \|_{H^{2s-r}(\mathbb{T}^2)}\leq\alpha^{-1}\| h\|_{H^r(\mathbb{T}^2)}$. With $h=g_N-E_N-\mathcal{I}^*g$ and $f^{\alpha,s}=\mathcal{I}^*g*p_{\alpha}^{s-r}$, this gives
    \begin{equation}
        \| (g_N-E_N)*p_{\alpha}^{s-r}-f^{\alpha,s} \|_{H^{2s-r}(\mathbb{T}^2)}\leq \frac{1}{\alpha}\| g_N-E_N-\mathcal{I}^*g \|_{H^r(\mathbb{T}^2)}\rightarrow0, \quad \text{as } N\rightarrow\infty.
    \end{equation}
    If $e\equiv 0$, then $E_N=0$ for every $N$, which gives the last claim and thus concludes the proof.

\end{proof}

From above result, we can estimate the error of the filtered torus backprojection.

\begin{corollary}[Error estimate]\label{cor:error_estimate} Under the assumptions of Theorem \ref{thm:filtered-TBP}, we have
\begin{equation}
    \lim_{N\rightarrow\infty}( \| f^{\alpha,s} - g_N * p_\alpha^{s-r} \|_{H^{2s-r}(\mathbb{T}^2)}- \| E_N * p_\alpha^{s-r} \|_{H^{2s-r}(\mathbb{T}^2)})=0.
\end{equation}
\end{corollary}
\begin{proof}
    By linearity, $g_N*p^{s-r}_\alpha - f^{\alpha,s} = E_N*p_{\alpha}^{s-r} + h_N$, where $h_N:=(g_N-E_N)*p_\alpha^{s-r} - f^{\alpha,s}$. By Theorem \ref{thm:filtered-TBP}, $h_N\rightarrow0$ in $H^{2s-r}(\mathbb{T}^2)$. Since $E_N\in H^{r}(\mathbb{T}^2)$, by \eqref{eq:filterupperbound}, $E_N*p_{\alpha}^{s-r}\in H^{2s-r}(\mathbb{T}^2)$. By the reverse triangle inequality, we have
    \begin{equation}
        \left| \| f^{\alpha,s} - g_N * p_\alpha^{s-r} \|_{H^{2s-r}(\mathbb{T}^2)} -   \| E_N * p_\alpha^{s-r} \|_{H^{2s-r}(\mathbb{T}^2)}\right| \leq \| h_N \|_{H^{2s-r}(\mathbb{T}^2)}\rightarrow0,
    \end{equation}
    which concludes the proof.
\end{proof}

\begin{corollary}[Noisy data]\label{cor:noisydata} Let $r\in\mathbb{R}$, $s\geq r$, $\alpha>0$ and $f\in H^r(\mathbb{T}^2)$. Suppose that $g=\mathcal{I}f + \eta$, where $\eta\in H^r(\mathbb{T}^2\times Q)$. Then $e$ and $E_N$ depend only on $\eta$. If $\eta$ is constant along geodesics, then 
\begin{equation}
\lim_{N\rightarrow\infty} g_N*p_\alpha^{s-r} = P^{s-r}_\alpha f + P^{s-r}_\alpha \mathcal{I}^*\eta \quad \text{in }H^{2s-r}(\mathbb{T}^2),
\end{equation}
and
\begin{equation}
    \| P^{s-r}_\alpha \mathcal{I}^*\eta \|_{H^{r}(\mathbb{T}^2)} \leq \| \eta \|_{H^r(\mathbb{T}^2 \times Q)}, \qquad \| P^{s-r}_\alpha \mathcal{I}^*\eta \|_{H^{2s-r}(\mathbb{T}^2)} \leq \frac{1}{\alpha} \| \eta \|_{H^{r}(\mathbb{T}^2 \times Q)}.
\end{equation}
\end{corollary}
\begin{proof}
    By \eqref{eq:periodicity_IF}, $\mathcal{I}f(\cdot,v)$ is constant along geodesics in direction $v$ and by Theorem \ref{thm:FourierSlice} $\widehat{\mathcal{I}f}(k,v)=0$ when $k\cdot v \neq 0$. Thus, $\mathcal{I}f$ does not contribute to $e$ or $E_N$. Since $\mathcal{I}f\in H^r(\mathbb{T}^2\times Q)$ \cite[Proposition 10]{IKR19}, we have $g\in H^r(\mathbb{T}^2\times Q)$. If $\eta$ is constant along geodesics, $e\equiv 0$. Theorem \ref{thm:filtered-TBP} gives $g_N*p_{\alpha}^{s-r}\rightarrow f^{\alpha,s}$ in $H^{2s-r}(\mathbb{T}^2)$. By linearity and $\mathcal{I}^*\mathcal{I}=\text{Id}$ \cite[Proposition 11]{IKR19} we get $f^{\alpha, s}=P^{s-r}_\alpha f + P^{s-r}_\alpha\mathcal{I}^* \eta$. Since $\| \mathcal{I} \| = \|\mathcal{I}^*\|=1$ \cite[Propositions 10 and 11]{IKR19}, the first inequality follows from $\|P^{s-r}_\alpha\|_{H^r\rightarrow H^r}\leq 1$, and the second from $\|P^{s-r}_\alpha\|_{H^r\rightarrow H^{2s-r}}\leq \alpha^{-1}$ by \eqref{eq:filterupperbound}.
\end{proof}

\subsection{Positivity constraint}
\label{subsec:positivity}
In standard X-ray attenuation imaging, the attenuation coefficient is non-negative by definition, which makes the positivity constraint a natural physical prior. We define the pointwise projection
  \begin{equation}
      f_+ := \max\left(f^{\alpha,s}_\mathrm{rec}, 0 \right).
  \end{equation}
If the ground truth $f_\mathrm{true}\in L^p(\mathbb{T}^2)$ satisfies $f_\mathrm{true}\geq 0$, then for any $1\leq p \leq \infty$ it holds that
 \begin{equation}
     \| f_\mathrm{true} - f_+ \|_{L^p(\mathbb{T}^2)} \leq \| f_\mathrm{true} - f^{\alpha,s}_{\mathrm{rec}} \|_{L^p(\mathbb{T}^2)},
 \end{equation}
since replacing negative values of $f_\mathrm{rec}$ by zero reduces the pointwise error whenever $f^{\alpha,s}_\mathrm{rec}<0$. In practice, $f_+$ is applied as a post-processing step after the regularized reconstruction $f^{\alpha,s}_\mathrm{rec}$.

\section{Numerical implementation}
\label{sec:implementation}

In the original work \cite{IKR19}, forward models for the geodesic X-ray transform on $\mathbb{T}^2$ were developed, including both analytical and torus-projection based approaches. In this study, we revisit these experiments using an updated implementation of the forward operator $\mathcal{A}_{\mathbb{T}^2}$, which removes certain numerical and simulation errors that were present in the previous work.

In addition, we introduce two extensions of the reconstruction framework. First, we extend the frequency coverage of the Fourier reconstruction. Second, we implement a numerical reconstruction method based on the torus backprojection formula, together with its regularized variants.

\subsection{Closed geodesics on $\mathbb{T}^2$}

Each periodic geodesic $\gamma_{x,v}$ on $\mathbb{T}^2$ can be represented as a finite disjoint union of line segments within the fundamental domain $[0,1)^2$. Denoting the segments by $v_i, i=1,...,N$, the X-ray integral along the geodesic is the sum of the integrals along its segments. Figure \ref{fig:segments} illustrates an example with a geodesic corresponding to $v=(3,2)$. This decomposition allows the use of the standard Radon transform values on $\mathbb{R}^2$ to compute the torus X-ray transform.

\begin{figure}[h]
    \centering
    \includegraphics[width=0.4\linewidth]{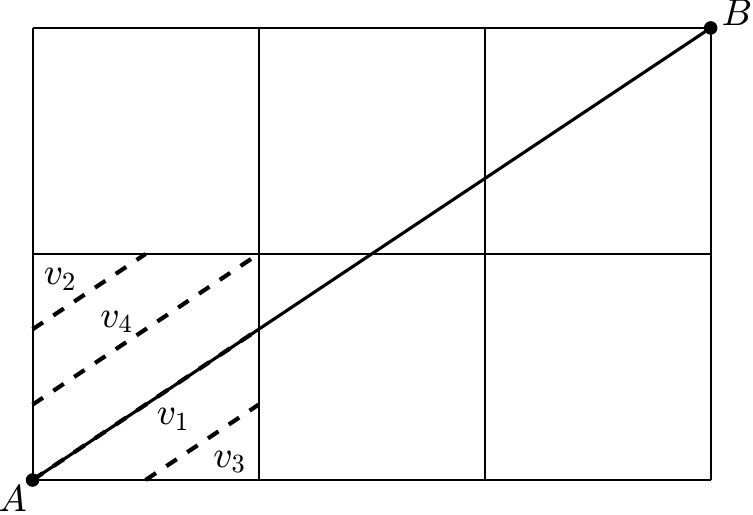}
    \caption{The closed geodesic on $\mathbb{T}^2$ determined by $v=(3,2)$, viewed in the fundamental domain $\left[ 0,1 \right)^2$ embedded in $\mathbb{R}^2$. The geodesic segment from $A$ to $B$ decomposes into four line segments $v_1,v_2,v_3$ and $v_4$ within $\left[ 0,1 \right)^2$. The torus X-ray transform along the geodesic is computed as a sum of the standard Radon transform values over these segments.}
    \label{fig:segments}
\end{figure}

\subsection{Discrete forward operator on the torus}
\label{subsec:forward}
Given a Radon transform data $Rf(v)$ for a set of discrete directions $v\in Q_N$, let the rays in direction $v$ be indexed by $k=1,\ldots, M$. Let $c_{k,v}$ denote the signed distance of the $k$'th ray from the origin and $Rf(v)_k$ the corresponding projection value. The geodesic is split into segments travelling from boundary to boundary of the fundamental domain $[0,1)^2$ and $d_{x,v,i}$ denotes the distance of the $i$'th segment of the geodesic from the origin. The torus forward operator is defined following \cite[Section 3.1.2]{IKR19} as
\begin{equation}
\label{eq:forwardop}
\mathcal{A}_{\mathbb{T}^2}f(x,v) = \frac{1}{|v|}\sum_{i=1}^N (w_{1,i}Rf(v)_{k_{1,i}} + w_{2,i}Rf(v)_{k_{2,i}}),
\end{equation}
where $k_{1,i}$ and $k_{2,i}$ are the indices of the two nearest rays to the segment $i$'th geodesic segment $v_i$, and $w_{1,i}, w_{2,i}$ are linear interpolation weights defined by
\begin{equation}
    w_{1,i} = \left|\frac{c_{k_{2,i},v}-d_{x,v,i}}{c_{k_{1,i},v}-c_{k_{2,i},v}}\right|, \quad w_{2,i}  = \left|\frac{c_{k_{1,i},v}-d_{x,v,i}}{c_{k_{1,i},v}-c_{k_{2,i},v}}\right|,
\end{equation}
when $|c_{k_{1,i},v} - d_{x,v,i}|+|c_{k_{2,i},v} - d_{x,v,i}| < |c_{k_{1,i},v} - c_{k_{2,i},v}|$ and $w_{1,i}=w_{2,i}=0$ otherwise. The zero-weight condition ensures that only rays on opposite sides of the geodesic segment contribute to the interpolation. 

The computation is repeated for each starting point $x\in X$, where $X$ denotes the set of geodesic starting points on the torus: $n_d$ equally spaced points on the $x$-axis, or on the $y$-axis for $v=(1,0)$. This discrete operator converts usual Radon transform data into X-ray data on the torus, preserving the correspondence with Fourier coefficients needed for reconstruction.

\subsection{Discrete inverse model}
Given the torus projection data $g=\mathcal{A}_{\mathbb{T}^2}f$ applied to a noisy sinogram, the Fourier coefficients $\hat{f}(k)$ are recovered for frequencies $k \in \mathbb{Z}^2$ contained in a finite box $\mathbb{Z}^2_N:=[-N,N]^2 \cap \mathbb{Z}^2$, using the inversion formula of Theorem \ref{thm:inversion}. The regularized reconstruction is obtained by applying the adjoint operator $\mathcal{I}^*$ followed by the Tikhonov filter $P^{s}_\alpha$. In all numerical experiments, we take $r=0$. Then we have
\begin{equation}
    f^{\alpha,s}_{\mathrm{rec}}(x)=\sum_{k\in\mathbb{Z}^2_N} \frac{1}{1+\alpha\langle k\rangle^{2s}}\,\hat{g}(k,\hat{k}^\perp)\,e^{2\pi i k \cdot x},
\end{equation}
where $\hat{g}(k,\hat{k}^\perp) = \widehat{I^*g}(k)$ is the Fourier coefficient of the adjoint applied to data, as given by equation \eqref{eq:adjoint} and $\hat{k}^\perp$ is the unique vector in $Q_N$ perpendicular to $k$.

\subsection{Extended frequency coverage -- Star TCT}

The standard Torus CT method recovers Fourier coefficients $\hat{f}(k)$ for frequencies $k\in\mathbb{Z}^2_N:=[-N,N]^2\cap\mathbb{Z}^2$ using measurement directions $v\in Q_N$. By the Fourier slice theorem (Theorem \ref{thm:FourierSlice}), the projection data in direction $v\in\mathbb{Z}^2\setminus\{0\}$ determines $\hat{f}(k)$ for any $k\perp v$. In particular, if $v\perp k$ then $v\perp mk$ for any $m\in\mathbb{Z}$ so the same projection data simultaneusly determines $\hat{f}(mk)$ for all $m\geq1$. We define the extended frequency set
\begin{equation}
    \mathcal{K}_{N,\widetilde{N}}:=\{ mk:k\in\mathbb{Z}^2_N,m\in\mathbb{Z}\setminus\{0\}, \| mk\|_\infty \leq \widetilde{N}\}
\end{equation}
The set $\mathcal{K}_{N,\widetilde{N}}$ is the collection of frequencies recoverable from the data at projection angles $Q_N$ within the box $\mathbb{Z}_{\widetilde{N}}$. These coefficients are depicted in Figure \ref{fig:star_coefficients}.
\begin{figure}[h]
    \centering
    \includegraphics[width=0.45\linewidth]{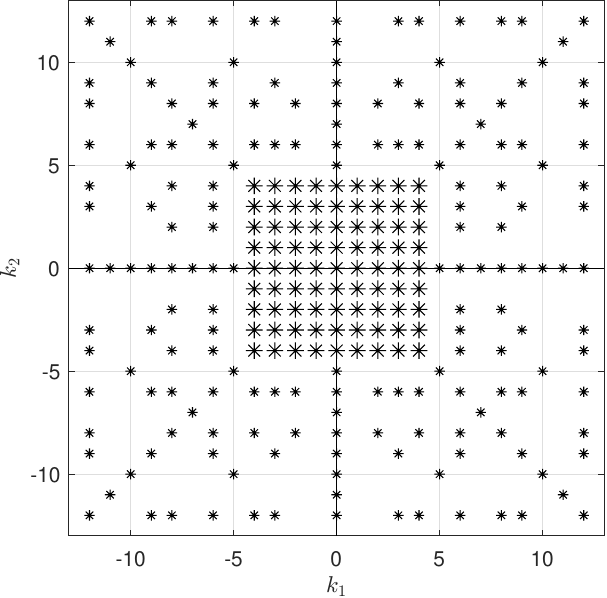}
    \caption{Coefficients in the extended frequency set $\mathcal{K}_{4, 12}$.}
    \label{fig:star_coefficients}
\end{figure}

The extended reconstruction is defined as the truncated Fourier series over $\mathcal{K}_{N,\widetilde{N}}$ with Tikhonov filter $P_\alpha^{s}$ as follows
\begin{equation}
     f^{\alpha,s}_{\mathrm{rec},*}(x)=\sum_{k\in\mathcal{K}_{N,\widetilde{N}}} \frac{1}{1+\alpha\langle k\rangle^{2s}}\,\hat{g}(k,\hat{k}^\perp)\,e^{2\pi i k \cdot x}.
\end{equation}
We refer to this variant as Star TCT.

\subsection{Discrete TBP}\label{subsec:directsum_numerical}

In the finite measurement setting, the sum in Theorem \ref{thm:TBP} is restricted to the discrete direction set $Q_N$. For each direction $v\in Q_N$, the torus X-ray data $g((x_j,0),v)$ are computed at the $x$-axis starting points $(x_j,0)\in X$ using the forward operator $\mathcal{A}_{\mathbb{T}^2}$ as described in Section \ref{subsec:forward}. Since the data are constructed geodesic by geodesic, we have $e\equiv 0$ in \eqref{eq:error_e} for exact data $g=\mathcal{I}f$ and for noisy or experimental data.

For noisy data, the averages $\bar{g}_v$ of $g(\cdot,v)$ over $\mathbb{T}^2$ differ slightly, so the data does not satisfy the equal-average condition of the data space $H^r(\mathbb{T}^2\times Q)$. Therefore, we replace $\hat{g}(0,0)$ in \eqref{eq:g_N_TBP} by the direction-wise averages and define
\begin{equation}
\label{eq:f_TBP}
    f_{\mathrm{TBP}}(x)=\sum_{v\in Q_N}(g(x,v) - \frac{|Q_N|-1}{|Q_N|}\bar{g}_v).
\end{equation}

The constants only affect the zero mode, which equals $\frac{1}{|Q_N|}\sum_{v\in Q_N}\bar{g}_v$. $f_\text{TBP}$ coincides with $g_N$ in \eqref{eq:g_N_TBP} with the average $\hat{g}(0,0)$ replaced by the common average. These modified data lie in $H^r(\mathbb{T}^2\times Q)$, have the same nonzero Fourier coefficients as $g$ and satisfy $e\equiv0$, so Theorem \ref{thm:filtered-TBP} applies with $E_N=0$. For exact data, all $\bar{g}_v$ equal $\hat{f}(0)$ and \eqref{eq:f_TBP} reduces to $f_N$ in \eqref{eq:f_N_TBP}. The effect of the averaging of the zero mode is estimated in Proposition \ref{prop:zeromode} below.

The value of $g(x,v)$ at a reconstruction point $x=(x_1,x_2)$ is obtained by identifying which geodesic in direction $v$ passes through $x$. For $v_2\neq0$, the starting point $(x_j,0)$ of such geodesic is given by 
\begin{equation}
    x_j = x_1-\frac{v_1}{v_2}x_2 \pmod{1},
\end{equation}
which follows from solving $x=(x_j,0)+tv$ for $x_j$. For $v =(1,0)$, the geodesic through $x$ is horizontal and starts at $(0,x_2)$ on the y-axis. Since $x_j$ does not, in general, coincide with one of the discrete starting points in $X$, the value $g((x_j,0),v)$ is obtained by linear interpolation of the data.

The filtered torus backprojection $f^{\alpha,s}_{\text{TBP}}=f_{\text{TBP}}*p_\alpha^s$ of Theorem~\ref{thm:filtered-TBP} is implemented numerically in two ways. In both cases, the filter is approximated by its truncated Fourier series over a finite box $\mathbb{Z}_M$,
\begin{equation}
    p_{\alpha,M}^{s}(x) = \sum_{k\in\mathbb{Z}^2_M} \frac{1}{1+\alpha\langle k\rangle^{2s}} e^{2\pi ik\cdot x}.
\end{equation}

\begin{enumerate}[label=(\roman*)]
\item \textbf{Fourier-domain implementation (fFTBP).} By the convolution theorem on $\mathbb{T}^2,$
\begin{equation}
    \widehat{f^{\alpha,s}_{\text{TBP}}}(k)=\frac{1}{1+\alpha\langle k\rangle^{2s}}\widehat{f_{\mathrm{TBP}}}(k).
\end{equation}
This is implemented by using the Fast Fourier transform (FFT) on $f_{\text{TBP}}$, multiplying pointwise by the filter coefficients, and applying the inverse FFT. 
    \item \textbf{Convolution implementation (cFTBP).} The convolution $f_{\text{TBP}}*p_\alpha^s$ is evaluated directly as
\begin{equation}
    f^{\alpha,s}_{\text{TBP}}(x) = \int_{\mathbb{T}^2}f_{\text{TBP}}(y)p_{\alpha,M}^s(x-y)\,dy, 
\end{equation}
approximated numerically over the reconstruction grid as a sum
\begin{equation}
    f_{\text{TBP}}^{\alpha,s}(x_{i_1,i_2})=\sum_{j_1,j_2} f_{\text{TBP}}(x_{j_1,j_2})p_{\alpha,M}^s(x_{i_1,i_2}-x_{j_1,j_2})h^2,
\end{equation}
where $h$ denotes the grid spacing and difference $x_{i_1,i_2}-x_{j_1,j_2}$ is taken modulo one. 

\end{enumerate}

Finally, we estimate the effect of the direction-wise averages in \eqref{eq:f_TBP} on the zero mode.
\begin{proposition}[Zero mode averaging]
\label{prop:zeromode}
    Let $g=If + \eta$ be a noisy data corresponding to some $f \in \mathcal{T}'$ with an error term $\eta: \mathbb{T}^2 \times Q \to \mathbb{R}$. Suppose that the errors $\mathcal{E}_v := \bar{g}_v-\hat{f}(0),$ for $v\in Q_N$, are real-valued random variables with zero mean, with variance at most $C \geq 0$, and are pairwise uncorrelated. Then the zero mode of the filtered torus backprojection is an unbiased estimator of $\hat{f}(0)/(1+\alpha)$. In particular,
    \begin{equation}
        \mathbb{E}\widehat{f^{\alpha,s}_{\mathrm{TBP}}}(0)=\frac{\hat{f}(0)}{1+\alpha},
    \end{equation}
    and its mean squared error satisfies
    \begin{equation}
    \mathbb{E}\left|\widehat{f^{\alpha,s}_{\mathrm{TBP}}}(0)-\frac{\hat{f}(0)}{1+\alpha}\right|^2\leq \frac{C}{(1+\alpha)^2|Q_N|}.
    \end{equation}
\end{proposition}
\begin{proof}
    The constant terms in \eqref{eq:f_TBP} give $\widehat{f_{\mathrm{TBP}}}(0)=\frac{1}{|Q_N|}\sum_{v\in Q_N}\bar{g}_v$, and since $\langle 0\rangle=1$, for the filtered case the zero mode is multiplied by $(1+\alpha)^{-1}$. Thus
    \begin{equation}
        \widehat{f^{\alpha,s}_{\mathrm{TBP}}}(0)-\frac{\hat{f}(0)}{1+\alpha} = \frac{1}{(1+\alpha)|Q_N|}\sum_{v\in Q_N}(\bar{g}_v-\hat{f}(0)).
    \end{equation}
    Taking expectations and using that the errors have zero mean gives the unbiasedness. Since the errors have zero mean and are pairwise uncorrelated, $\mathbb{E}(\bar{g}_v - \hat{f}(0))\overline{(\bar{g}_w-\hat{f}(0))}=0$ for $v\neq w$, $v,w \in Q_N$, and since $\mathbb{E}|\bar{g}_v - \hat{f}(0)|^2\leq C$, we get
    \begin{equation}
        \mathbb{E}\left| \ \sum_{v\in Q_N}(\bar{g}_v-\hat{f}(0)) \right|^2
 =  \sum_{v\in Q_N}\mathbb{E} \left| \bar{g}_v-\hat{f}(0)\right|^2 \leq C|Q_N|.   \end{equation}
 Now dividing by $(1+\alpha)^2|Q_N|^2$ gives the mean squared error estimate.
\end{proof}

Since the estimator is unbiased, its mean squared error equals its variance, so the bound in Proposition \ref{prop:zeromode} is also a bound for $\mathrm{Var}\widehat{f_{\mathrm{TBP}}^{\alpha,s}}(0)$. By \cite[Proposition 16]{IKR19}, the following estimate $$|Q_N| = \frac{2}{\zeta(2)}N^2 + O(N\log N)$$ holds. Hence, the bound decays like $N^{-2}$. Estimating $\hat{f}(0)$ from only one direction as in \eqref{eq:mean}, gives only the bound $C/(1+\alpha)^2$. The direction-wise averaging reduces the mean squared error by a factor of $|Q_N|$. 

\subsection{Numerical setup}

The implementation builds on the publicly available code \cite{KR19} accompanying \cite{IKR19}, and has been adapted and extended for this work to include parallel-beam-to-torus mapping, Star TCT, filtered TBP, and real-data experiments.

The results are compared with filtered backprojection (FBP) using (i) the same set of directions as in the torus reconstruction and (ii) the same number of uniformly distributed angles. The positivity constraint defined in Section \ref{subsec:positivity} is applied as a post-processing step to all methods to evaluate its effect on the reconstruction results.

The regularization parameters $\alpha$ and $s$ are optimized separately for both cases. Relative reconstruction errors
\begin{equation}
    \epsilon_p=\frac{\| f_{\mathrm{true}} - f_{\mathrm{rec}} \|_{L^p(\mathbb{T}^2)}}{\| f_{\mathrm{true}} \|_{L^p(\mathbb{T}^2)}}
\end{equation}
are evaluated in $L^p(\mathbb{T}^2)$ for $p\in \{1,2,\infty \}$ both with and without applying the positivity constraint.

\section{Experiments with simulated data}
\label{sec:simulated}
The numerical experiments are performed using the same $256\times256$ resolution phantoms as in \cite{IKR19}. The measurements corresponding to the forward operator \eqref{eq:forwardop} are simulated by using measurements $Rf(v)_k$ created by \texttt{radon.m} (MATLAB R2024b, MathWorks Inc.) with added zero-mean Gaussian noise with a relative level of $\sigma=0.02$. The FBP reference data are generated using a higher-resolution phantom and mapped to the reconstruction grid by interpolation, following the no-crime data generation approach of Mueller and Siltanen \cite{MS12}. All reconstructions are performed using a Fourier coefficient box $\mathbb{Z}^2_N$ with $N=50$. The corresponding set of torus measurement directions $Q_N$ contains $3096$ unique directions, and the set of geodesic starting points $X$ is chosen to consist of 256 equally spaced points.

The evaluated methods include Torus CT, the extended variant Star TCT with enlarged frequency set, torus backprojection (TBP) using both the unregularized and the two different regularization methods (cFTBP and fFTBP), and filtered backprojection, computed using either the torus-optimal directions $Q_N$ or the same number of uniformly distributed angles. Reconstruction errors for all evaluated methods without the positivity constraint are shown in Table \ref{tab:synthetic}, and the corresponding results with positivity constraint in Table \ref{tab:synthetic_pos_full}. 

The regularization parameters $\alpha$ and $s$ are selected independently for each phantom and method by minimizing the $L^2$ reconstruction error over a two-dimensional grid of $\alpha\in[0,10^{-4}]$ and $s\in[0,2]$. The chosen parameters are presented in Table \ref{tab:regparams}, and the corresponding error surfaces are shown in supplementary material \cite{GithubSupplement}. For the TBP, both regularization implementation variants yield the same results, so they are reported only once.

\begin{table}[h]
\centering

\begin{tabular}{c | cc | cc | cc}
\hline
\multicolumn{7}{c}{Unconstrained} \\
\hline
 & \multicolumn{2}{c}{Shepp--Logan} 
 & \multicolumn{2}{c}{Flag} 
 & \multicolumn{2}{c}{Rot. Flag} \\
\hline
 & $\alpha$ & $s$ & $\alpha$ & $s$ & $\alpha$ & $s$ \\
\hline

Torus CT 
& $5.0\cdot10^{-6}$& $1.35$& $2.5\cdot10^{-5}$& $1.50$& $4.5\cdot10^{-5}$& $1.40$\\

Star TCT 
& $5.0\cdot10^{-6}$& $1.40$& $4.5\cdot10^{-5}$& $1.40$& $3.5\cdot10^{-5}$& 1.45\\

f/cFTBP 
& $5.0\cdot10^{-6}$& $1.40$& $2.5\cdot10^{-5}$& $1.50$& $3.5\cdot10^{-5}$& 1.45\\

\hline

\multicolumn{7}{c}{With positivity constraint $f_+ = \max(f,0)$} \\
\hline
 & \multicolumn{2}{c}{Shepp--Logan} 
 & \multicolumn{2}{c}{Flag} 
 & \multicolumn{2}{c}{Rot. Flag} \\
\hline
 & $\alpha$ & $s$ & $\alpha$ & $s$ & $\alpha$ & $s$ \\
\hline

Torus CT 
& $5.0\cdot10^{-6}$& $1.25$& $5.0\cdot10^{-6}$& $1.65$& $1.0\cdot10^{-5}$& $1.55$\\

Star TCT 
& $5.0\cdot10^{-6}$& $1.35$& $1.5\cdot10^{-5}$& $1.50$& $1.5\cdot10^{-5}$& 1.50\\

f/cFTBP 
& $5.0\cdot10^{-6}$& $1.35$& $5.0\cdot10^{-6}$& $1.65$& $1.5\cdot10^{-5}$& 1.50\\

\hline
\end{tabular}

\caption{Regularization parameters $\alpha$ and $s$ selected by minimizing the $L^2$ error for each phantom.}
\label{tab:regparams}
\end{table}

Figure \ref{fig:Torus CT_synthetic_N50} shows the Torus CT reconstructions using the updated codes. 

\begin{figure}[tbp]
    \centering

    \begin{subfigure}{0.25\textwidth}
        \centering
        \includegraphics[width=\linewidth]{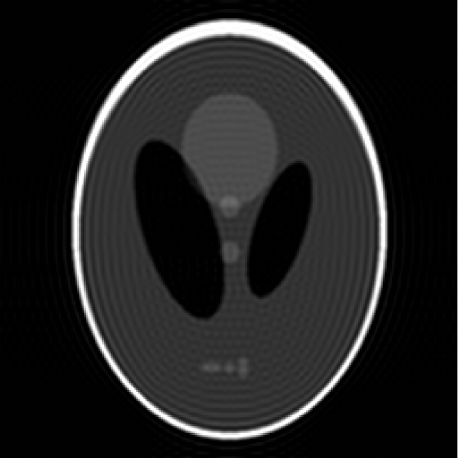}
        \caption{}
    \end{subfigure}
    \hspace{0.05\textwidth}
    \begin{subfigure}{0.25\textwidth}
        \centering
        \includegraphics[width=\linewidth]{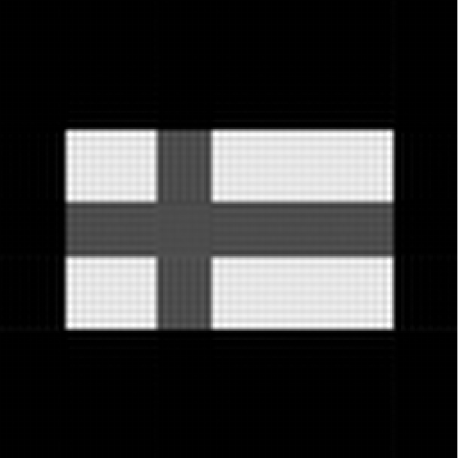}
        \caption{}
    \end{subfigure}
    \hspace{0.05\textwidth}
    \begin{subfigure}{0.25\textwidth}
        \centering
        \includegraphics[width=\linewidth]{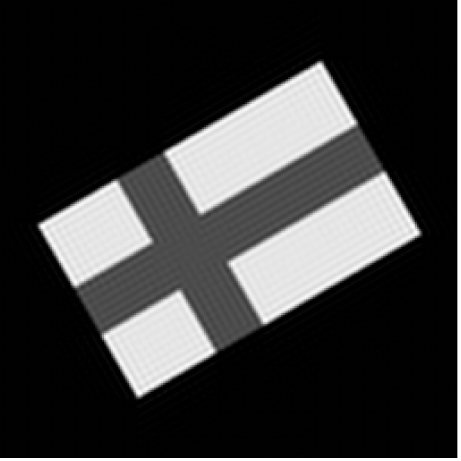}
        \caption{}
    \end{subfigure}

    \vspace{4mm}

    \begin{subfigure}{0.25\textwidth}
        \centering
        \includegraphics[width=\linewidth]{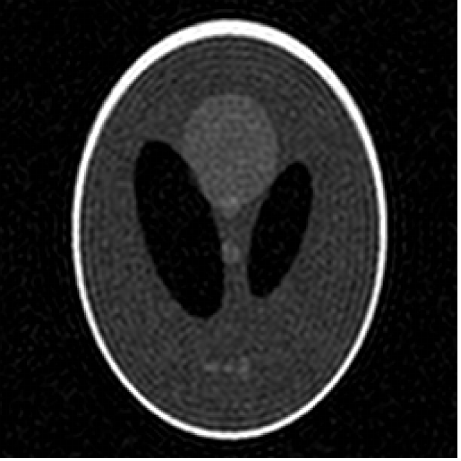}
        \caption{}
    \end{subfigure}
    \hspace{0.05\textwidth}
    \begin{subfigure}{0.25\textwidth}
        \centering
        \includegraphics[width=\linewidth]{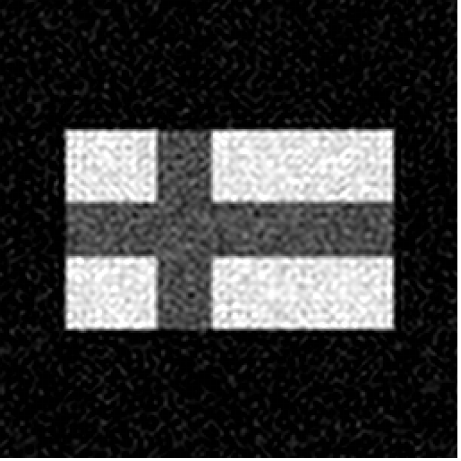}
        \caption{}
    \end{subfigure}
    \hspace{0.05\textwidth}
    \begin{subfigure}{0.25\textwidth}
        \centering
        \includegraphics[width=\linewidth]{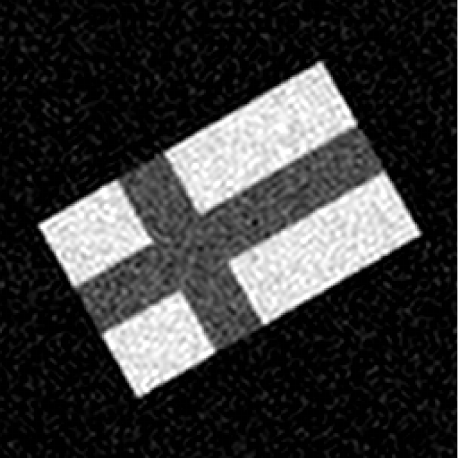}
        \caption{}
    \end{subfigure}

    \vspace{4mm}

    \begin{subfigure}{0.25\textwidth}
        \centering
        \includegraphics[width=\linewidth]{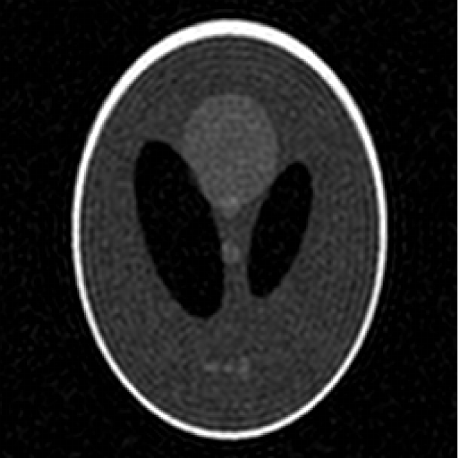}
        \caption{}
    \end{subfigure}
    \hspace{0.05\textwidth}
    \begin{subfigure}{0.25\textwidth}
        \centering
        \includegraphics[width=\linewidth]{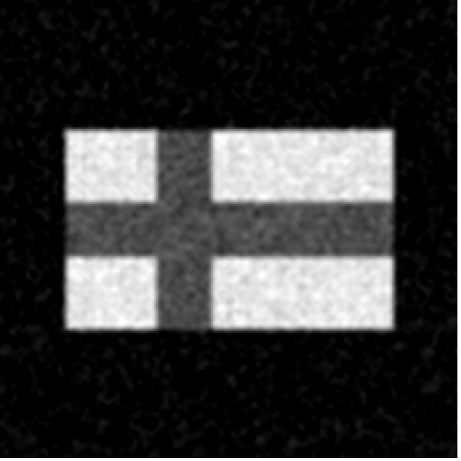}
        \caption{}
    \end{subfigure}
    \hspace{0.05\textwidth}
    \begin{subfigure}{0.25\textwidth}
        \centering
        \includegraphics[width=\linewidth]{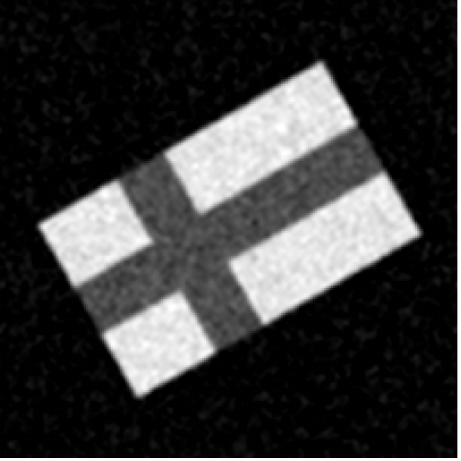}
        \caption{}
    \end{subfigure}

    \caption{Torus CT reconstructions for three phantoms (Shepp--Logan, flag, and flag rotated by $30^\circ$) with box size $N=50$. Top row (A)-(C): reconstructions from noise-free data. Middle row (D)-(F): reconstructions from noisy data. Bottom row (G)-(I): reconstructions from noisy data with regularization. Figures are shown in $\left[ 0,1 \right]$ for printability.}
    \label{fig:Torus CT_synthetic_N50}
\end{figure}

\begin{table}[tbp]
    \centering
    \setlength{\tabcolsep}{4pt}
    \resizebox{\textwidth}{!}{%
    \begin{tabular}{c c | c c c | c c c | c c c}
        & &\multicolumn{3}{c|}{Noiseless} & \multicolumn{3}{c|}{Noisy, no regularization} & \multicolumn{3}{c}{Noisy, regularized} \\
        \hline
         & & S-L & Flag & Rot.\ Flag & S-L & Flag & Rot.\ Flag & S-L & Flag & Rot.\ Flag \\
         \hline
          \multirow{3}{*}{Torus CT}  & $\epsilon_1$ & 18.8\%& 8.49\%& 7.52\%& 29.9\%& 27.0\%& 26.7\%& 27.9\%& 17.2\%& 17.5\%\\
           & $\epsilon_2$ & 26.1\%& 12.4\%& 11.6\%& 28.4\%& 19.4\%& 19.1\%& 28.1\%& 15.8\%& 15.7\%\\
           & $\epsilon_\infty$ & 76.5\%& 81.9\%& 70.8\%& 76.5\%& 71.2\%& 79.6\%& 73.8\%& 70.6\%& 69.4\%\\
          \hline
          \multirow{3}{*}{Star TCT}  & $\epsilon_1$ & 20.0\%& 8.01\%& 8.22\%& 35.6\%& 35.5\%& 35.6\%& 28.1\%& 17.5\%& 17.4\%\\
           & $\epsilon_2$ & 24.9\%& 11.4\%& 10.9\%& 29.5\%& 23.9\%& 24.0\%& 27.8\%& 15.5\%& 15.7\%\\
           & $\epsilon_\infty$ & 87.4\%& 89.3\%& 77.2\%& 90.8\%& 85.4\%& 90.2\%& 80.1\%& 71.2\%& 69.6\%\\
          \hline
          \multirow{3}{*}{TBP} & $\epsilon_1$ & 19.5\%& 7.34\%& 7.97\%& 37.6\%& 38.7\%& 37.8\%& 28.9\%& 18.4\%& 18.5\%\\
           &  $\epsilon_2$ & 24.9\%& 11.4\%& 10.9\%& 30.5\%& 25.9\%& 25.9\%& 28.3\%& 16.2\%& 16.2\%\\
           & $\epsilon_\infty$ & 84.6\%& 93.3\%& 78.0\%& 86.8\%& 92.2\%& 93.4\%& 79.6\%& 71.5\%& 69.1\%\\
          \hline
          \multirow{3}{*}{FBP with $Q_N$}&  $\epsilon_1$ & 26.2\%& 26.7\%& 19.0\%& 29.3\%& 29.9\%& 23.4\%& 
& & \\
          &  $\epsilon_2$ & 26.3\%& 19.3\%& 14.2\%& 27.0\%& 20.9\%& 16.4\%& & & \\
          &  $\epsilon_\infty$ & 75.1\%& 62.2\%& 75.6\%& 77.7\%& 66.2\%& 80.5\%& & & \\
          \hline
          \multirow{3}{*}{FBP uniform} & $\epsilon_1$ & 10.1\%& 2.72\%& 3.58\%& 16.9\%& 14.1\%& 14.7\%& & & \\
           & $\epsilon_2$ & 20.0\%& 6.12\%& 8.37\%& 20.9\%& 10.2\%& 11.7\%& & & \\
           & $\epsilon_\infty$ & 70.6\%& 39.2\%& 69.4\%& 69.9\%& 44.3\%& 74.1\%& & & \\
    \end{tabular}}
    \caption{Reconstruction errors for $N=50$ with noiseless and noisy data ($\sigma = 0.02$). In Star TCT, $\widetilde{N}=100$. All unregularized reconstructions use $\alpha=0, s=0$. Regularization parameters are selected independently for each phantom by minimizing the $L^2$ error. The regularization parameters used in each case are shown in Table \ref{tab:regparams}. For TBP, the regularized results show the same regularized value acquired by both implementations, cFTBP and fFTBP.}
    \label{tab:synthetic}
\end{table}

\begin{table}[tbp]
    \centering
    \setlength{\tabcolsep}{4pt}
    \resizebox{\textwidth}{!}{%
    \begin{tabular}{c  c | c c c | c c c | c c c}
        & &\multicolumn{3}{c|}{Noiseless} & \multicolumn{3}{c|}{Noisy, no regularization} & \multicolumn{3}{c}{Noisy, regularized} \\
        \hline
         &  & S-L & Flag & Rot.\ Flag & S-L & Flag & Rot.\ Flag & S-L & Flag & Rot.\ Flag \\
         \hline
         
          \multirow{3}{*}{Torus CT} & $\epsilon_1$ 
          & 16.3\%& 7.38\%& 6.44\%& 23.7\%& 19.2\%& 18.7\%& 23.0\%& 14.0\%& 13.8\%\\
          
          &  $\epsilon_2$ 
          & 25.7\%& 12.3\%& 11.5\%& 27.4\%& 17.3\%& 16.9\%& 27.4\%& 15.1\%& 14.9\%\\
          
          &  $\epsilon_\infty$ 
          & 76.5\%& 81.9\%& 70.9\%& 76.5\%& 71.3\%& 79.6\%& 75.1\%& 70.7\%& 69.6\%\\
          
          \hline
          
          \multirow{3}{*}{Star TCT} & $\epsilon_1$ 
          & 16.4\%& 6.12\%& 6.30\%& 27.2\%& 24.2\%& 24.2\%& 23.2\%& 13.9\%& 13.9\%\\
          
          &  $\epsilon_2$ 
          & 24.5\%& 11.0\%& 10.7\%& 27.9\%& 20.4\%& 20.4\%& 26.9\%& 14.8\%& 14.8\%\\
          
          &  $\epsilon_\infty$ 
          & 87.4\%& 89.3\%& 77.2\%& 90.8\%& 85.4\%& 90.2\%& 82.4\%& 71.4\%& 69.5\%\\
          
          \hline
          
          \multirow{3}{*}{TBP}  
          & $\epsilon_1$ 
          & 16.2\%& 5.73\%& 6.20\%& 28.8\%& 26.3\%& 26.3\%& 23.8\%& 14.6\%& 14.6\%\\

          & $\epsilon_2$ 
          & 24.6\%& 11.2\%& 10.7\%& 28.7\%& 21.9\%& 21.9\%& 27.4\%& 15.3\%& 15.3\%\\

          & $\epsilon_\infty$ 
          & 84.6\%& 93.3\%& 78.0\%& 86.8\%& 92.2\%& 93.4\%& 81.7\%& 72.0\%& 70.7\%\\

          \hline
          
          \multirow{3}{*}{FBP with $Q_N$} &  $\epsilon_1$ 
          & 19.7\% & 24.0\% & 13.7\%
          & 22.0\%& 25.8\% & 16.5\%
          & & & \\
          
          &  $\epsilon_2$ 
          & 24.6\% & 18.9\% & 12.4\%
          & 25.2\%& 20.2\% & 14.3\%
          & & & \\
          
          &  $\epsilon_\infty$ 
          & 75.1\% & 62.2\% & 75.6\%
          & 75.7\%& 66.2\%& 80.5\%
          & & & \\
          
          \hline
          
          \multirow{3}{*}{FBP uniform} & $\epsilon_1$ 
          & 8.74\% & 2.39\% & 3.16\%
          & 13.6\% & 9.64\% & 10.4\%
          & & & \\
          
          &  $\epsilon_2$ 
          & 19.9\% & 6.10\% & 8.35\%
          & 20.6\% & 8.95\% & 10.7\%
          & & & \\
          
          &  $\epsilon_\infty$ 
          & 70.6\% & 39.2\% & 69.4\%
          & 69.9\% & 44.1\% & 74.1\%
          & & & \\
          
    \end{tabular}}
    \caption{Reconstruction errors for $N=50$ with noiseless and noisy data ($\sigma = 0.02$) with positivity constraint $f_+ = \max(f,0)$. Regularization parameters are selected independently for each phantom by minimizing the $L^2$ error. The regularization parameters used in each case are shown in Table \ref{tab:regparams}. For TBP, the regularized results show the same regularized value acquired by both implementations, cFTBP and fFTBP.}
    \label{tab:synthetic_pos_full}
\end{table}

Star TCT reconstructions are computed with the same initial box size of $N=50$, while extending the reconstruction domain to $\widetilde{N}=100$. Due to the extended frequency coverage, the method yields slightly improved reconstruction accuracy in the noiseless and regularized cases compared to the standard Torus CT. The corresponding reconstructions are shown in Figure \ref{fig:StarTCT_synthetic_N50}.

\begin{figure}[tbp]
    \centering

    \begin{subfigure}{0.25\textwidth}
        \centering
        \includegraphics[width=\linewidth]{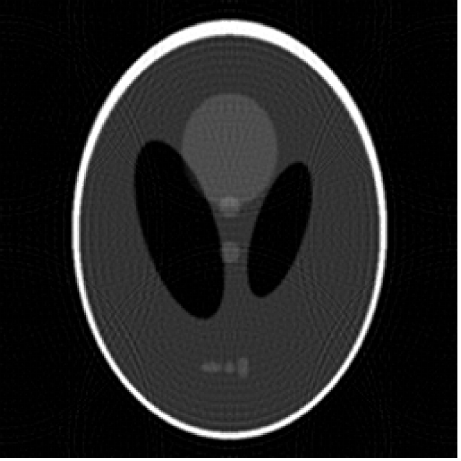}
        \caption{}
    \end{subfigure}
    \hspace{0.05\textwidth}
    \begin{subfigure}{0.25\textwidth}
        \centering
        \includegraphics[width=\linewidth]{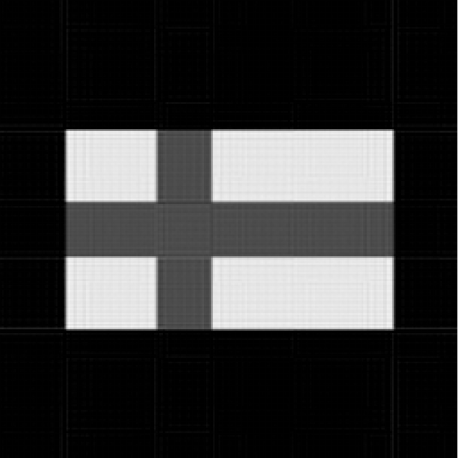}
        \caption{}
    \end{subfigure}
    \hspace{0.05\textwidth}
    \begin{subfigure}{0.25\textwidth}
        \centering
        \includegraphics[width=\linewidth]{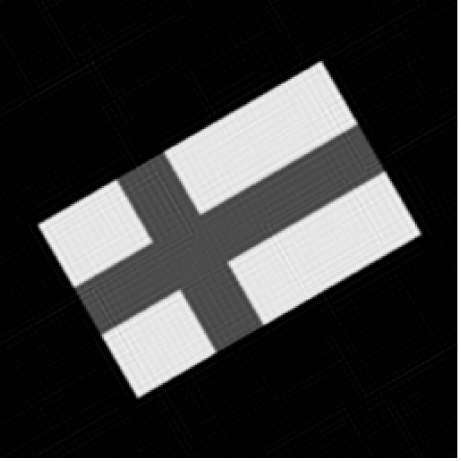}
        \caption{}
    \end{subfigure}

    \vspace{4mm}

    \begin{subfigure}{0.25\textwidth}
        \centering
        \includegraphics[width=\linewidth]{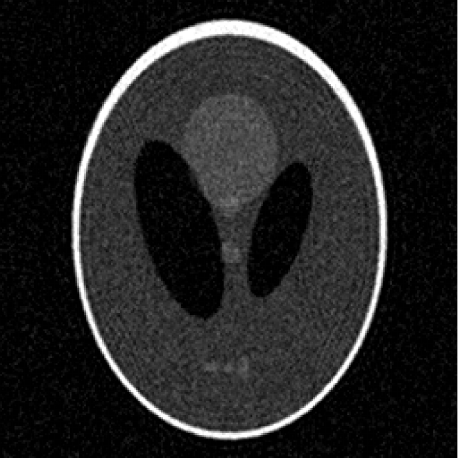}
        \caption{}
    \end{subfigure}
    \hspace{0.05\textwidth}
    \begin{subfigure}{0.25\textwidth}
        \centering
        \includegraphics[width=\linewidth]{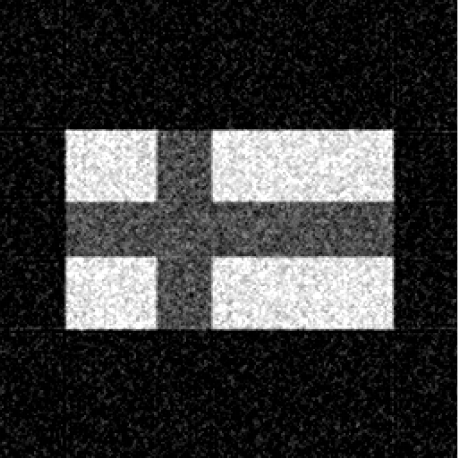}
        \caption{}
    \end{subfigure}
    \hspace{0.05\textwidth}
    \begin{subfigure}{0.25\textwidth}
        \centering
        \includegraphics[width=\linewidth]{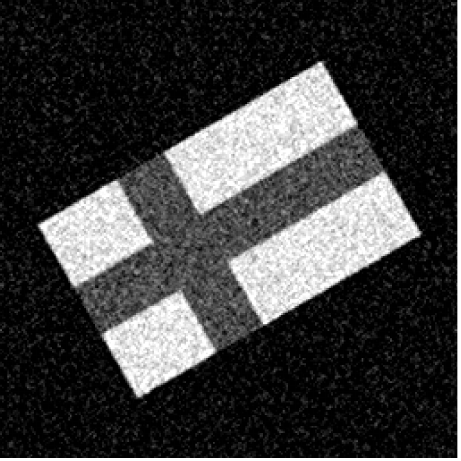}
        \caption{}
    \end{subfigure}

    \vspace{4mm}

    \begin{subfigure}{0.25\textwidth}
        \centering
        \includegraphics[width=\linewidth]{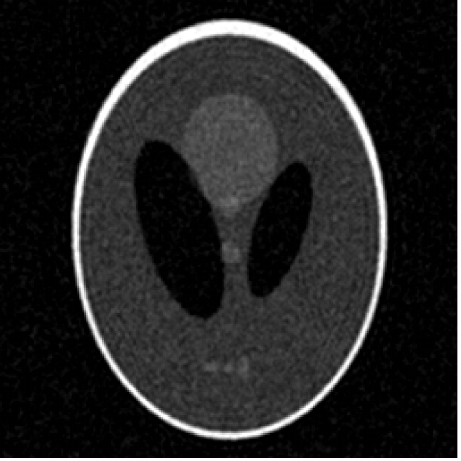}
        \caption{}
    \end{subfigure}
    \hspace{0.05\textwidth}
    \begin{subfigure}{0.25\textwidth}
        \centering
        \includegraphics[width=\linewidth]{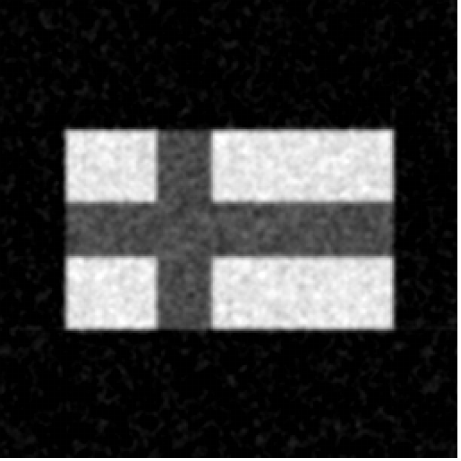}
        \caption{}
    \end{subfigure}
    \hspace{0.05\textwidth}
    \begin{subfigure}{0.25\textwidth}
        \centering
        \includegraphics[width=\linewidth]{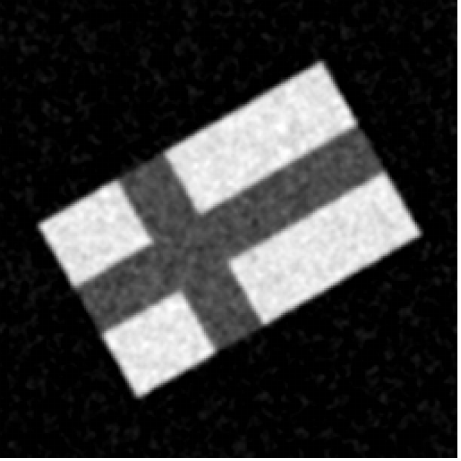}
        \caption{}
    \end{subfigure}

    \caption{Star TCT reconstructions for the three phantoms with $N=50$ and $\widetilde{N}=100$. Top row (A)-(C): reconstructions from noise-free data. Middle row (D)-(F): reconstructions from noisy data. Bottom row (G)-(I): reconstructions from noisy data with regularization. Figures are shown in $\left[ 0,1 \right]$ for printability.}
    \label{fig:StarTCT_synthetic_N50}
\end{figure}

TBP reconstructions are illustrated in Figure \ref{fig:TBP_synthetic_N50}. In the noiseless setting, TBP achieves the lowest $L^2$-errors among the torus-based methods. In the noisy data reconstructions, however, TBP obtains higher error norms than the other torus-based methods. Both regularization implementations yield the same error norm values. However, the Fourier-space implementation is computationally much less expensive.

\begin{figure}[tbp]
    \centering

    \begin{subfigure}{0.25\textwidth}
        \centering
        \includegraphics[width=\linewidth]{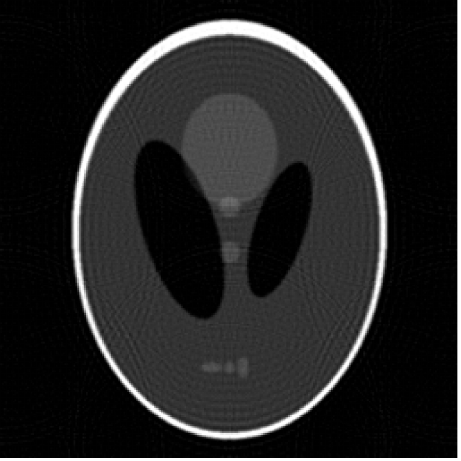}
        \caption{}
    \end{subfigure}
    \hspace{0.05\textwidth}
    \begin{subfigure}{0.25\textwidth}
        \centering
        \includegraphics[width=\linewidth]{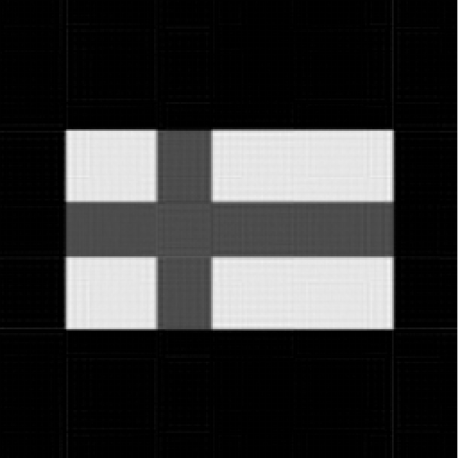}
        \caption{}
    \end{subfigure}
    \hspace{0.05\textwidth}
    \begin{subfigure}{0.25\textwidth}
        \centering
        \includegraphics[width=\linewidth]{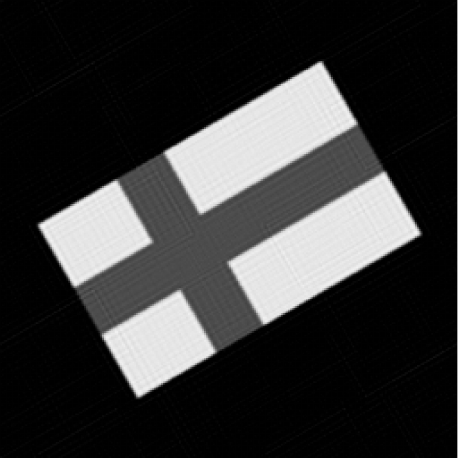}
        \caption{}
    \end{subfigure}

    \vspace{4mm}

    \begin{subfigure}{0.25\textwidth}
        \centering
        \includegraphics[width=\linewidth]{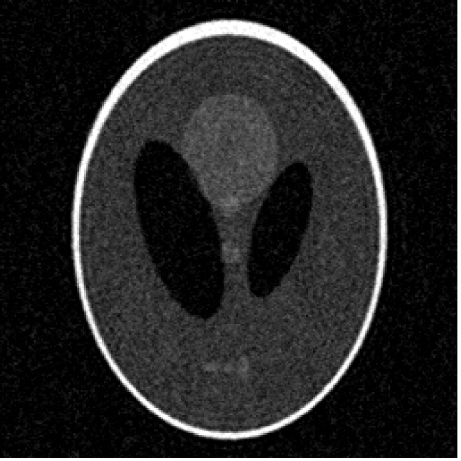}
        \caption{}
    \end{subfigure}
    \hspace{0.05\textwidth}
    \begin{subfigure}{0.25\textwidth}
        \centering
        \includegraphics[width=\linewidth]{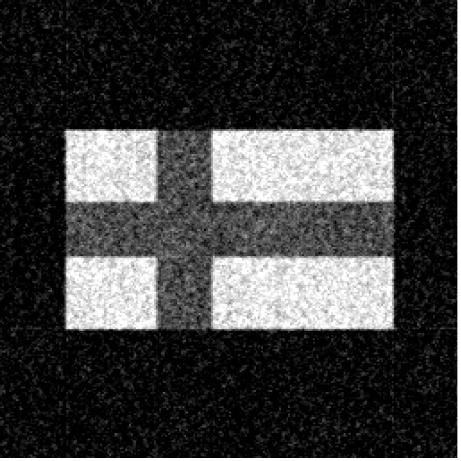}
        \caption{}
    \end{subfigure}
    \hspace{0.05\textwidth}
    \begin{subfigure}{0.25\textwidth}
        \centering
        \includegraphics[width=\linewidth]{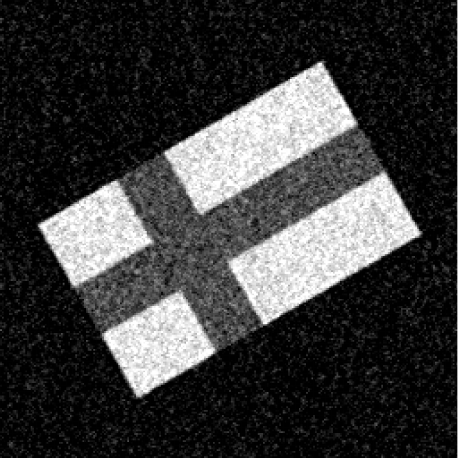}
        \caption{}
    \end{subfigure}

    \vspace{4mm}

        \begin{subfigure}{0.25\textwidth}
        \centering
        \includegraphics[width=\linewidth]{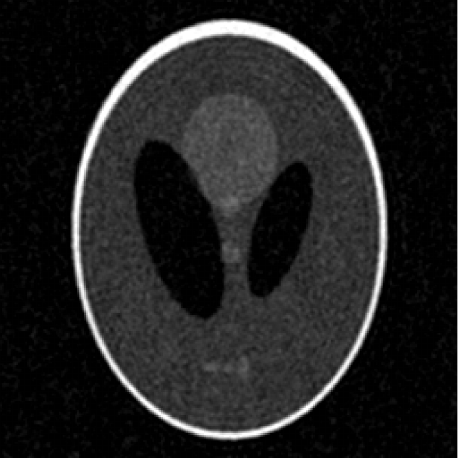}
        \caption{}
    \end{subfigure}
    \hspace{0.05\textwidth}
    \begin{subfigure}{0.25\textwidth}
        \centering
        \includegraphics[width=\linewidth]{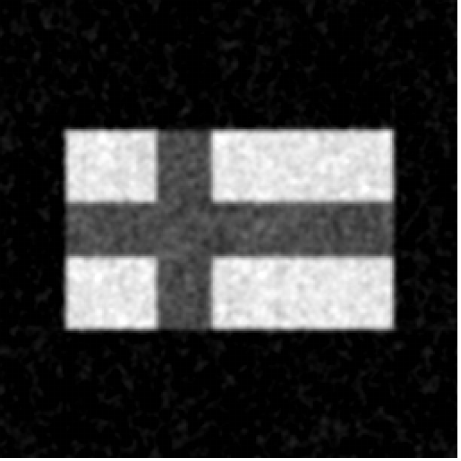}
        \caption{}
    \end{subfigure}
    \hspace{0.05\textwidth}
    \begin{subfigure}{0.25\textwidth}
        \centering
        \includegraphics[width=\linewidth]{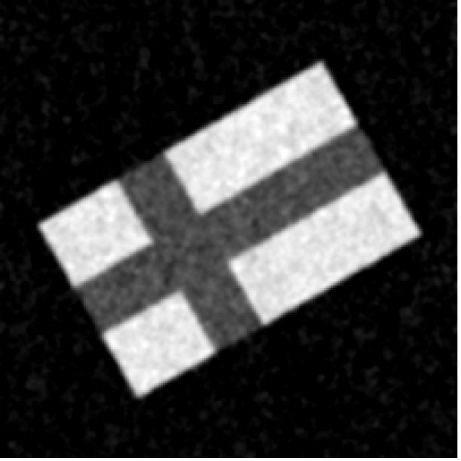}
        \caption{}
    \end{subfigure}

    \caption{TBP reconstructions for the three phantoms with $N=50$. Top row (A)-(C): reconstructions from noise-free data. Middle row (D)-(F): reconstructions from noisy data. Bottom row (G)-(I): reconstructions from noisy data with regularization using the Fourier-domain implementation (fFTBP). Figures are shown in $\left[ 0,1 \right]$ for printability.}
    \label{fig:TBP_synthetic_N50}
\end{figure}

Torus CT, Star TCT, and TBP result in smaller $L^2$ error norms than FBP with the torus-optimal angles $Q_N$ in the noiseless setting with the flag phantoms. Regularized Torus CT and Star TCT yield more accurate reconstructions than FBP with the same angle set for the flag phantoms. FBP also produces shadow-like artifacts when the projection directions are $Q_N$. These artifacts do not show up in torus-based methods. However, FBP with uniformly distributed angles yields the smallest errors among all evaluated methods. FBP reconstructions from both angle sets can be seen in Figure \ref{fig:fbpN50}

\begin{figure}[tbp]
    \centering

    \begin{subfigure}{0.25\textwidth}
        \centering
        \includegraphics[width=\linewidth]{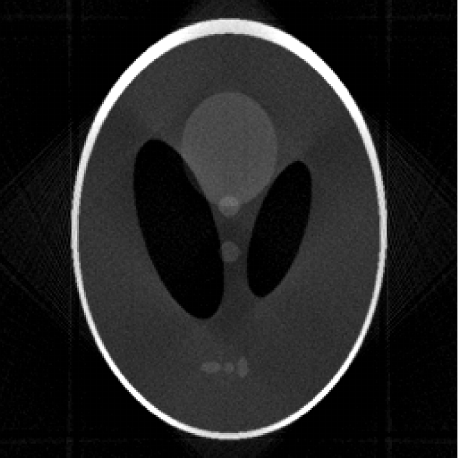}
        \caption{}
    \end{subfigure}
    \hspace{0.05\textwidth}
    \begin{subfigure}{0.25\textwidth}
        \centering
        \includegraphics[width=\linewidth]{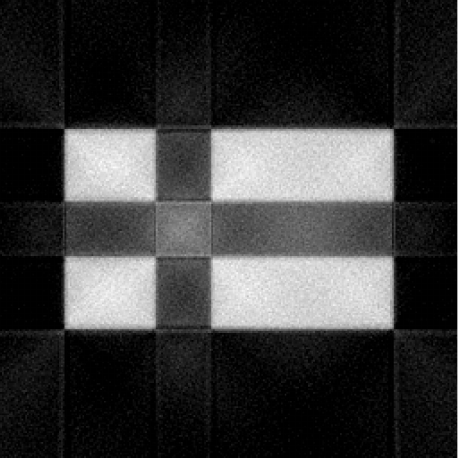}
        \caption{}
    \end{subfigure}
    \hspace{0.05\textwidth}
    \begin{subfigure}{0.25\textwidth}
        \centering
        \includegraphics[width=\linewidth]{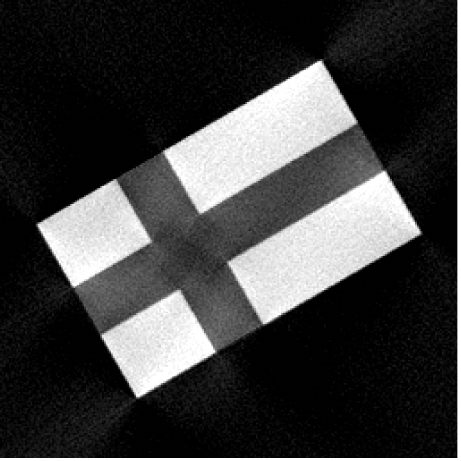}
        \caption{}
    \end{subfigure}

    \vspace{4mm}

    \begin{subfigure}{0.25\textwidth}
        \centering
        \includegraphics[width=\linewidth]{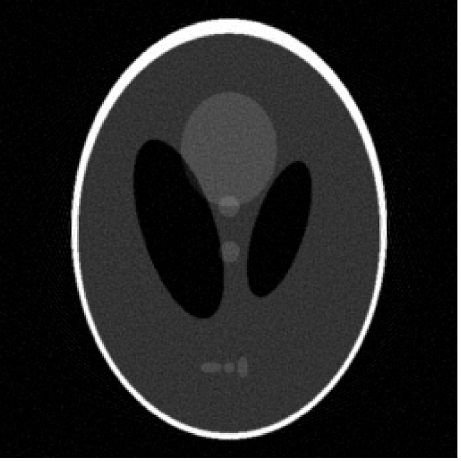}
        \caption{}
    \end{subfigure}
    \hspace{0.05\textwidth}
    \begin{subfigure}{0.25\textwidth}
        \centering
        \includegraphics[width=\linewidth]{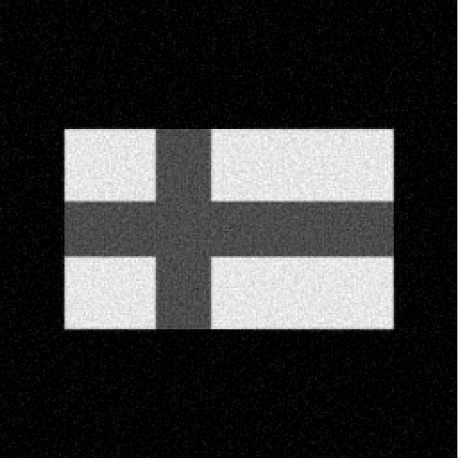}
        \caption{}
    \end{subfigure}
    \hspace{0.05\textwidth}
    \begin{subfigure}{0.25\textwidth}
        \centering
        \includegraphics[width=\linewidth]{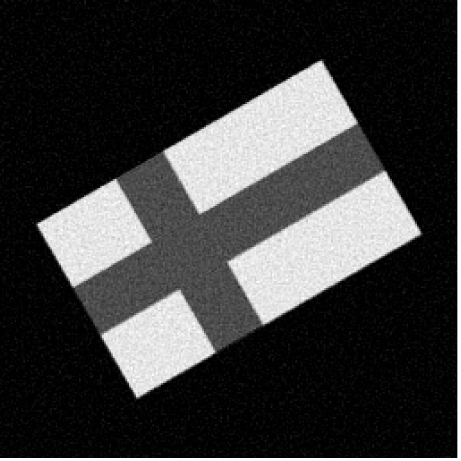}
        \caption{}
    \end{subfigure}

    \caption{Filtered backprojection reconstructions of noisy data. Top row (A)-(C): reconstructions from the set of directions $Q_N$. Bottom row (D)-(F): reconstructions from equally spaced angles. Figures are shown in $\left[ 0,1 \right]$ for printability.}
    \label{fig:fbpN50}
\end{figure}

Applying the positivity constraint $f_+=\max (f_\mathrm{rec},0)$ as a post-processing step further improves reconstruction quality for all methods in $L^1$ and $L^2$, as reported in Table \ref{tab:synthetic_pos_full}.

\section{Experiments with real data}
\label{sec:realdata}

In this section, we apply the Torus CT, the Star TCT and the TBP methods to experimental X-ray tomography data. The dataset consists of fan-beam measurements acquired at equally spaced projection angles, whereas the forward model $\mathcal{A}_{\mathbb{T}^2}$ operates on data taken in the torus geodesic directions $Q_N$. We therefore convert the fan-beam sinogram to parallel-beam Radon data and approximate the torus geodesic directions by selecting the closest available projection angle from the sinogram, as described in Section \ref{subsec:fan2torus}. The FBP reconstruction computed from all available projection angles of the converted sinogram is used as a reference against which reconstruction errors are computed. All methods are also compared with FBP using the same torus-optimal angles $Q_N$, and the positivity constraint of Section \ref{subsec:positivity} is applied as a post-processing step.

\subsection{Dataset description}
The experimental data used in this study consist of tomographic X-ray measurements of a walnut from a publicly available dataset \cite{walnut,walnut_arxiv}. For our reconstruction experiments, we use the sinogram \texttt{sinogram1200}, a matrix of size $2296\times1200$ stored in \texttt{FullSizeSinograms.mat}. The data were acquired using a fan-beam setting, with the object rotated around its axis in increments of $0.3^\circ$, covering a full $360^\circ$ rotation. The initial measurement corresponds to a top-down orientation. For full details of the measurement geometry, we refer to \cite{walnut_arxiv}.

The raw sinogram is converted from 16-bit unsigned integers to double precision. Each projection direction is normalized by the mean intensity of the background region to account for beam fluctuations, and the negative logarithm is taken to obtain line integrals of the X-ray attenuation. Minor misalignments of the rotation center are corrected using a circular shift operation.

\subsection{Fan-beam to torus projection mapping}
\label{subsec:fan2torus}
The raw X-ray data are acquired in a fan-beam geometry. To use the existing Torus CT forward model $\mathcal{A}_{\mathbb{T}^2}$, the fan-beam sinogram is first converted to parallel-beam Radon data using the $\texttt{fan2para.m}$ transformation (MATLAB R2024b, MathWorks Inc.). For each $v\in Q_N$, the closest measured angle from the Radon sinogram is selected. The resulting sinogram serves as an input for the discrete forward operator $\mathcal{A}_{\mathbb{T}^2}$. The angular resolution of $0.3^\circ$ ensures that the approximation error from nearest-angle selection remains small.

The set of torus measurement directions $Q_N$ contains $3096$ directions for $N=50$, while the converted parallel-beam sinogram provides only $600$ unique projection angles. As a consequence, each sinogram projection direction is selected as the closest match for multiple torus directions. In particular, the assumption of Proposition \ref{prop:zeromode} does not hold for directions that share a measured angle. While this means that the reconstruction does not exploit the correct number of independent measurements, the reuse of projection data affects all torus-based methods equally, preserving the validity of the relative comparison between Torus CT, Star TCT, TBP, and FBP with $Q_N$.

\subsection{Reconstructions and comparison to FBP}

The reference reconstruction is the FBP solution from all $600$ projection angles of the converted sinogram, shown in Figure \ref{fig:FBPgroundtruth}. As both Torus CT and the reference reconstruction are computed from the parallel-beam sinogram obtained through the \texttt{fan2para.m} transformation, any errors introduced during this conversion affect all methods equally and do not influence the relative comparison. 

The reconstructions are calculated for $N \in \{25, 50, 75,100\}$ and the number of starting points of the geodesics is set to $512$. 

\begin{figure}[h]
    \centering
    \includegraphics[width=0.4\linewidth]{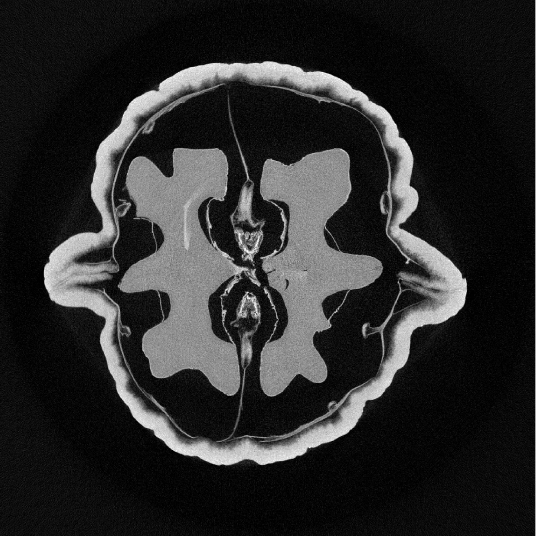}
    \caption{Reference reconstruction computed by FBP from all $600$ projection angles of the parallel-beam sinogram. This reconstruction is used as the reference in Tables \ref{tab:walnut} and \ref{tab:walnut_pos}. Figure is shown in $\left[ 0,10^{-3} \right]$ for printability.}
    \label{fig:FBPgroundtruth}
\end{figure}

Reconstruction errors for $N \in \{25, 50, 75,100\}$ are reported in Tables \ref{tab:walnut} and \ref{tab:walnut_pos} and illustrated in Figure \ref{fig:walnut_errors}. Reconstructions for $N=50$ and $N=100$ are shown in Figures \ref{fig:walnutN50} and \ref{fig:walnut100} respectively. For all methods, the reconstruction errors decrease as $N$ increases. Among the torus-based methods, TBP achieves the best accuracy with Star TCT approaching the same accuracy as $N$ and $\widetilde{N}$ increase. 

No regularization parameters were found that lower the $L^2$ error norms. This is most likely due to the interpolation errors and the reuse of measured data, which dominate over noise in this setting.

In comparison to FBP with torus-optimal angles $Q_N$, all torus-based methods exhibit higher $L^2$ errors. However, we note that the reference reconstruction is computed using FBP, which may introduce a systematic bias in the error metrics in its favor. The FBP reconstructions also exhibit the same shadow-like artifacts as in the simulated data. These artifacts are not present in the torus-based methods.

\begin{figure}[h]
    \centering

    \begin{subfigure}{0.32\textwidth}
        \centering
        \includegraphics[width=\linewidth]{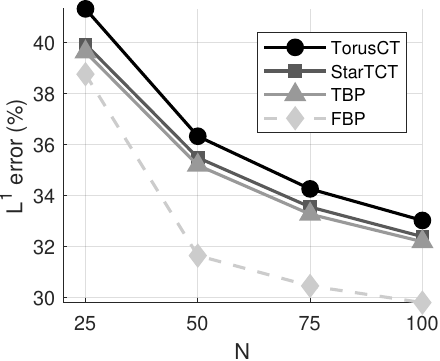}
        \caption{}
    \end{subfigure}
    \hfill
    \begin{subfigure}{0.32\textwidth}
        \centering
        \includegraphics[width=\linewidth]{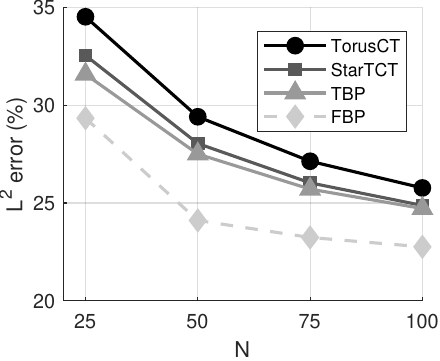}
        \caption{}
    \end{subfigure}
    \hfill
    \begin{subfigure}{0.32\textwidth}
        \centering
        \includegraphics[width=\linewidth]{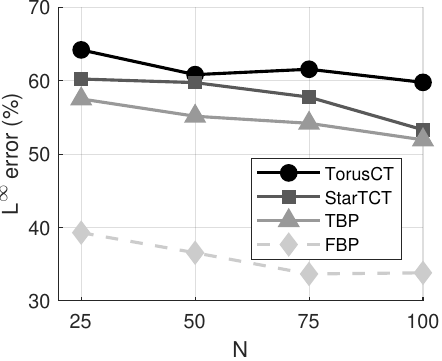}
        \caption{}
    \end{subfigure}

    \caption{Reconstruction errors $\epsilon_1$, $\epsilon_2$ and $\epsilon_\infty$ for Torus CT, Star TCT, TBP and FBP with torus-optimal angles $Q_N$ as functions of the Fourier coefficient box size $N$, computed from the walnut data set.}
    \label{fig:walnut_errors}
\end{figure}

\begin{table}[h]
    \centering
    \begin{tabular}{cc|c|c|c|c}
        & &Torus CT & Star TCT ($\widetilde{N}=2N$) & TBP & FBP from $Q_N$\\
        \hline
        \multirow{3}{*}{$N=25$}  
        & $\epsilon_1$       & 41.3\%& 39.9\%& 39.6\%& 38.8\%\\
        & $\epsilon_2$       & 34.5\%& 32.5\%& 31.6\%& 29.3\%\\
        & $\epsilon_\infty$  & 64.2\%& 60.2\%& 57.5\%& 39.3\%\\
        \hline
        \multirow{3}{*}{$N=50$}  
        & $\epsilon_1$       & 36.3\%& 35.5\%& 35.2\%& 31.6\%\\
        & $\epsilon_2$       & 29.4\%& 28.0\%& 27.5\%& 24.1\% \\
        & $\epsilon_\infty$  & 60.8\%& 59.7\%& 55.1\%& 36.6\% \\
        \hline
        \multirow{3}{*}{$N=75$}  
        & $\epsilon_1$       & 34.3\%& 33.5\%& 33.3\%& 30.5\%\\
        & $\epsilon_2$       & 27.1\%& 26.0\%& 25.7\%& 23.2\%\\
        & $\epsilon_\infty$  & 61.5\%& 57.8\%& 54.2\%& 33.7\%\\
        \hline
        \multirow{3}{*}{$N=100$} 
        & $\epsilon_1$       & 33.0\%& 32.4\%& 32.2\%& 29.8\%\\
        & $\epsilon_2$       & 25.8\%& 24.9\%& 24.7\%& 22.8\% \\
        & $\epsilon_\infty$  & 59.8\%& 53.3\%& 51.9\%& 33.8\% \\
    \end{tabular}
    \caption{Reconstruction errors for the walnut dataset. Errors are computed relative to the FBP reconstruction 
    from all $600$ projection angles of the parallel-beam sinogram.}
    \label{tab:walnut}
\end{table}

\begin{table}[h]
    \centering
    \begin{tabular}{cc|c|c|c|c}
        & & Torus CT & Star TCT ($\widetilde{N}=2N$) & TBP & FBP from $Q_N$ \\
        \hline
        \multirow{3}{*}{$N=25$}  
        & $\epsilon_1$       & 40.7\%& 38.3\%& 37.6\%& 35.7\%\\
        & $\epsilon_2$       & 34.2\%& 31.7\%& 30.4\%& 27.2\%\\
        & $\epsilon_\infty$  & 64.2\%& 60.2\%& 57.5\%& 35.9\%\\
        \hline
        \multirow{3}{*}{$N=50$}  
        & $\epsilon_1$       & 35.9\%& 34.3\%& 33.9\%& 30.6\%\\
        & $\epsilon_2$       & 29.2\%& 27.3\%& 26.7\%& 23.3\%\\
        & $\epsilon_\infty$  & 60.8\%& 59.7\%& 55.1\%& 31.6\%\\
        \hline
        \multirow{3}{*}{$N=75$}  
        & $\epsilon_1$       & 33.8\%& 32.6\%& 32.2\%& 29.7\%\\
        & $\epsilon_2$       & 26.9\%& 25.4\%& 25.1\%& 22.6\%\\
        & $\epsilon_\infty$  & 61.5\%& 57.8\%& 54.2\%& 31.4\%\\
        \hline
        \multirow{3}{*}{$N=100$} 
        & $\epsilon_1$       & 32.6\%& 31.5\%& 31.3\%& 29.2\%\\
        & $\epsilon_2$       & 25.5\%& 24.3\%& 24.2\%& 22.3\%\\
        & $\epsilon_\infty$  & 59.8\%& 53.3\%& 51.9\%& 32.3\% \\
    \end{tabular}
    \caption{Reconstruction errors for the walnut dataset with positivity constraint $f_+ = \max(f, 0)$ applied as a post-processing step. Errors are computed relative to the positivity-constrained FBP reconstruction from all $600$ projection angles of the parallel-beam sinogram.}
    \label{tab:walnut_pos}
\end{table}

\begin{figure}[h]
    \centering

    \begin{subfigure}{0.35\textwidth}
        \centering
        \includegraphics[width=\linewidth]{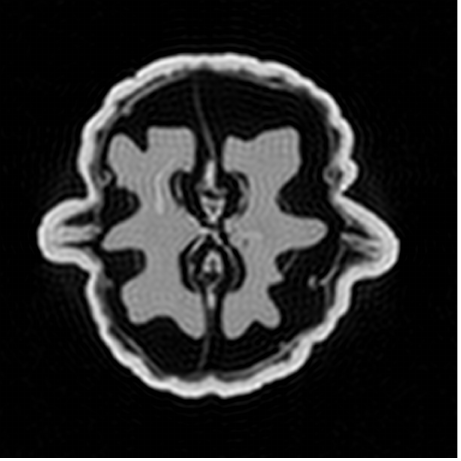}
        \caption{Torus CT}
    \end{subfigure}
    \hspace{0.10\textwidth}
    \begin{subfigure}{0.35\textwidth}
        \centering
        \includegraphics[width=\linewidth]{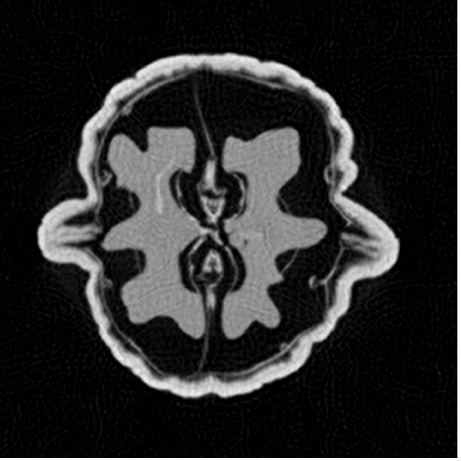}
        \caption{Star TCT}
    \end{subfigure}

    \begin{subfigure}{0.35\textwidth}
        \centering
        \includegraphics[width=\linewidth]{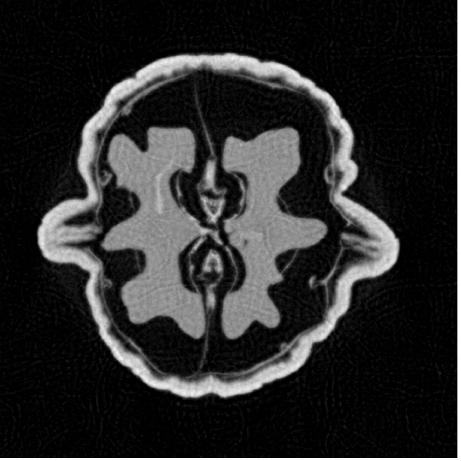}
        \caption{TBP}
    \end{subfigure}
    \hspace{0.10\textwidth}
    \begin{subfigure}{0.35\textwidth}
        \centering
        \includegraphics[width=\linewidth]{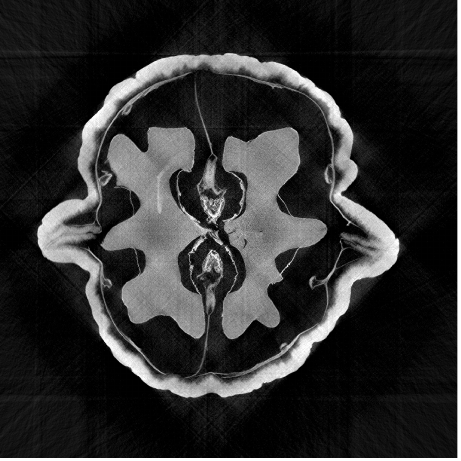}
        \caption{FBP with angles $Q_N$}
    \end{subfigure}

    \caption{Walnut reconstructions with $N=50$. Figures are shown in $\left[ 0,10^{-3} \right]$ for printability.}
    \label{fig:walnutN50}
\end{figure}

\begin{figure}[h]
    \centering

    \begin{subfigure}{0.35\textwidth}
        \centering
        \includegraphics[width=\linewidth]{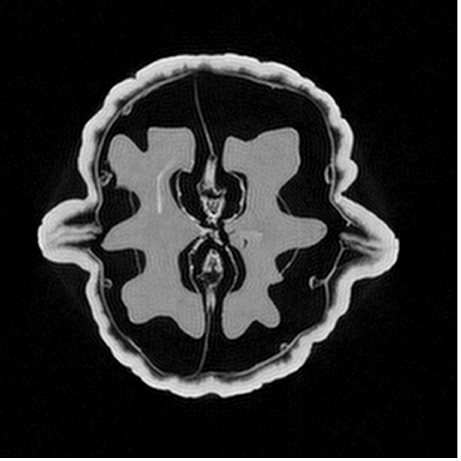}
        \caption{Torus CT}
    \end{subfigure}
    \hspace{0.10\textwidth}
    \begin{subfigure}{0.35\textwidth}
        \centering
        \includegraphics[width=\linewidth]{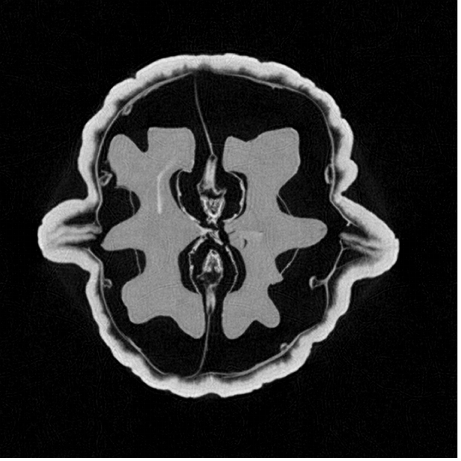}
        \caption{Star TCT}
    \end{subfigure}

    \begin{subfigure}{0.35\textwidth}
        \centering
        \includegraphics[width=\linewidth]{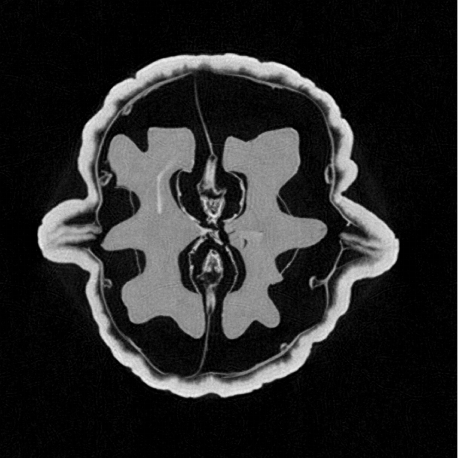}
        \caption{TBP}
    \end{subfigure}
    \hspace{0.10\textwidth}
    \begin{subfigure}{0.35\textwidth}
        \centering
        \includegraphics[width=\linewidth]{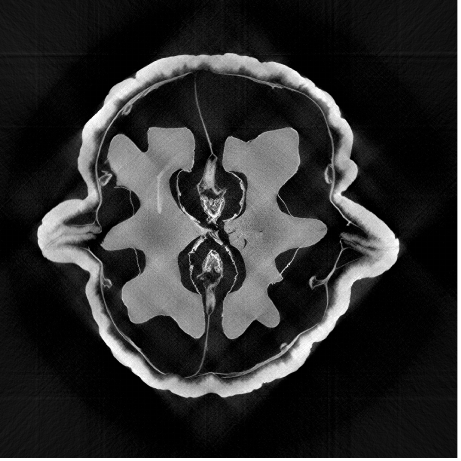}
        \caption{FBP with angles $Q_N$}
    \end{subfigure}

    \caption{Walnut reconstructions with $N=100$. Figures are shown in $\left[ 0,10^{-3} \right]$ for printability.}
    \label{fig:walnut100}
\end{figure}

The computing times of the torus based methods increase substantially with $N$ as shown in Figure \ref{fig:walnut_runtimes}. The dominant computational cost is the sinogram to torus mapping step, which increases from approximately $230$ s for $N=25$ to $13800$ s ($\approx 3.8$ h) for $N=100$. Once the mapping is complete, the reconstruction times vary substantially by method: TBP is the fastest at $18-28$ s, then Torus CT at $49-985$ s and Star TCT at $108-2814$ s. The computations were performed on a Dell OptiPlex SFF 7020 with an Intel Core i5-14500 processor and 32 GB of RAM running MATLAB R2024b. The implementation, however, has not been fully optimized for speed, and significant reductions in computation time could be achieved through, for example, simple parallelization.

\begin{figure}[h]
    \centering

    \begin{subfigure}{0.32\textwidth}
        \centering
        \includegraphics[width=\linewidth]{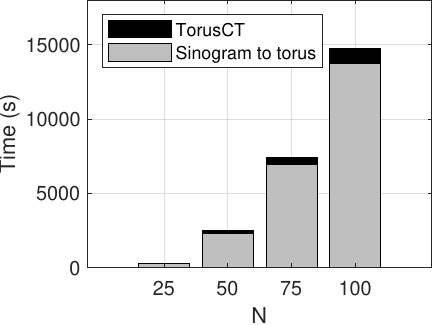}
        \caption{}
    \end{subfigure}
    \hfill
    \begin{subfigure}{0.32\textwidth}
        \centering
        \includegraphics[width=\linewidth]{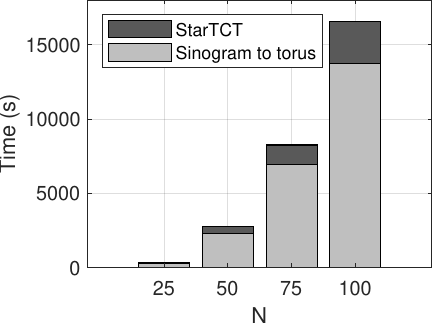}
        \caption{}
    \end{subfigure}
    \hfill
    \begin{subfigure}{0.32\textwidth}
        \centering
        \includegraphics[width=\linewidth]{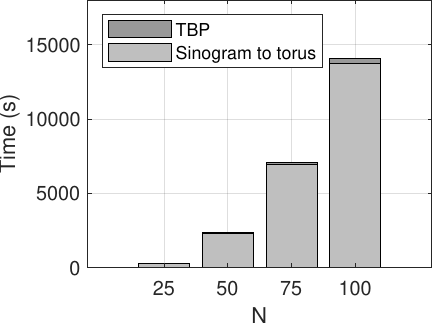}
        \caption{}
    \end{subfigure}

    \caption{Computing times in seconds for Torus CT (A), Star TCT (B), and TBP (C) as functions of $N$. Each bar shows the total time decomposed into the sinogram-to-torus mapping and the reconstruction step. The mapping step dominates the total computation and grows from approximately $230$ s  ($\approx 3.8$ min) for $N=25$ to $13800$ s ($\approx 3.8$ h) for $N=100$.  FBP is excluded as its computing time is negligible at this scale.}
    \label{fig:walnut_runtimes}
\end{figure}

\section{Conclusions}
\label{sec:conclusions}
In this study, we applied the Torus CT method proposed by \cite{IKR19} to experimental tomographic data and extended the numerical implementation to handle fan-beam datasets by converting them to parallel-beam format and mapping the projection angles to the geodesic directions on the torus. In addition, we introduced two extensions of the original method: the Star TCT reconstruction with extended frequency coverage, and the torus backprojection (TBP) based on the summation of the projection data. Additionally, we evaluated a pointwise positivity constraint as a new post-processing step and reran the simulated data experiments from \cite{IKR19} using the updated code.

On the theory side, we showed that the filtered torus backprojection converges in $H^{2s-r}(\mathbb{T}^2)$ to the Tikhonov regularized solution for general data, up to an explicit error term that measures how far the data are from being constant along geodesics (Theorem \ref{thm:filtered-TBP}). Because data measured geodesic by geodesic, including both noisy and experimental data, are constant along geodesics, this error term vanishes in practice. For noisy data, the limit is the regularized image plus a filtered noise term bounded by the noise level (Corollaries \ref{cor:error_estimate} and \ref{cor:noisydata}). We also showed that for uncorrelated errors, the direction-wise averaging used to estimate the zero mode is unbiased and reduces its mean squared error by a factor of $|Q_N|$ compared to using a single direction. 

For simulated data, the revised implementation yielded better results than those reported in \cite{IKR19}. All the torus-based methods result in a better reconstruction in the noiseless case than FBP computed from the same angles in $L^2$. Among the torus-based methods, Star TCT improves on standard Torus CT in noiseless and regularized cases, while TBP achieves the lowest errors in the noiseless setting but is more sensitive to noise. Both implementations of the regularized TBP produce identical error values, with the Fourier side filtering (fFTBP) being computationally more efficient than the numerical integral method (cFTBP). In the noisy regularized case, all torus-based methods resulted in lower $L^2$-errors than FBP with torus-optimal angles for the flag-type phantoms. The standard FBP with uniformly distributed angles, however, remains the most accurate overall.

The positivity constraint reduces the $L^1$ and $L^2$ errors for all methods, with a more pronounced effect for Fourier-based reconstructions. This is consistent with the nature of the truncated Fourier series, which produces Gibbs-type oscillations that produce larger negative values than the artifacts typical of FBP, and are therefore more effectively suppressed by the positivity constraint. This effect is less pronounced in the regularized reconstructions. 

For real data, the reconstruction methods were applied to fan-beam X-ray measurements of a walnut \cite{walnut}. We note that, since the reference reconstruction is computed using FBP, the error metrics may give FBP with torus-optimal angles an advantage over the torus-based methods, regardless of reconstruction quality. Considering this, the reconstruction errors are comparable to those of FBP with the same torus-optimal angles $Q_N$. Applying the positivity constraint reduces the errors for both methods but does not substantially change the relative difference. 

In the walnut experiments, the torus directions are filled from only $600$ measured angles, while the number of torus angles exceeds that. The errors of directions that share a measured angle are thus strongly correlated, so the reduction of zero-mode error predicted by Proposition \ref{prop:zeromode} is limited by the number of distinct measured angles rather than by $|Q_N|$.

The FBP reconstructions exhibit shadow-like streak artifacts along the diagonal, horizontal, and vertical directions due to the non-uniform angular distribution of $Q_N$ \cite{Q93, Q17}, whereas the Torus CT reconstructions are free of these artifacts. For small $N$, the Torus CT reconstructions exhibit oscillatory Gibbs-type artifacts characteristic of truncated Fourier series, which decrease as $N$ increases.

The computational cost of mapping the sinogram to the torus increases significantly with $N$, and is a computational bottleneck for the applicability of these methods. Once this step is completed, however, the reconstruction itself is relatively inexpensive. The meshless nature of the torus-based reconstructions allows evaluation on arbitrary grids without additional cost.

The fan-beam-to-parallel-beam conversion and the nearest-angle selection introduce approximation errors in the torus projection mapping, which are likely the dominant sources of error in the real-data experiments and are not addressed by the Tikhonov regularization. A more accurate treatment could involve interpolating directly in the fan-beam geometry or acquiring measurements at the geodesic angles $Q_N$ and thus avoiding the need for nearest-angle selection.

Future work includes testing Torus CT on a wider range of phantoms to better characterize the conditions under which it outperforms FBP. Experimentation with additional real datasets would also provide further evidence of the method's practical applicability. A natural direction is to acquire experimental data directly along the torus geodesic directions $Q_N$ with a dense detector to avoid the approximation error introduced by nearest-angle selection and decrease the error caused by interpolating on the torus. 

\subsection*{MATLAB codes and data} 

The MATLAB implementation accompanying this work is publicly available on GitHub:
\url{https://github.com/ella-salo/TorusCT}.

An archived version of the code for the simulated data experiments is available on Zenodo \cite{SKR26}.

The experimental walnut data set used in this work is publicly available at \cite{walnut}.

\subsection*{Acknowledgements} E.S. and J.R. were supported by the Research Council of Finland through the Flagship of Advanced Mathematics for Sensing, Imaging and Modelling (grant number 359183). A.M. was supported by Business Finland (decision no.
8132/31/2022), and by the Research Council of Finland through the Centre of Excellence in Inverse Modelling and Imaging (grant no. 353097). J.R. was additionally supported by the Emil Aaltonen Foundation, and the Jenny and Antti Wihuri Foundation.

\subsection*{Declaration of generative AI use} During the preparation of this work, the authors used Claude Opus 5/5.5 to improve the final presentation. The authors reviewed and verified the AI-assisted content.

\newpage

\clearpage
\bibliographystyle{abbrv}
\bibliography{sample}

\end{document}